\documentclass[12pt]{amsart}
\usepackage[utf8]{inputenc}
\usepackage{amsmath}
\usepackage{amssymb}
\usepackage{xcolor,float}
\usepackage[normalem]{ulem}
\usepackage{tikz}
\newcommand{\vertex}{\node[vertex]}
\usetikzlibrary{patterns}

\usepackage{subfig}

\usepackage{bm}

\usepackage{amsthm}
\usepackage{centernot}
\usepackage[top=1in, bottom=1in, left=1in, right=1in]{geometry}

\usepackage{enumitem, hyperref}
\makeatletter
\def\namedlabel#1#2{\begingroup
    #2%
    \def\@currentlabel{#2}%
    \phantomsection\label{#1}\endgroup
}

\newcommand{\code}{\calC}

\newcommand{\realiz}{\calU}

\newcommand{\fullreal}{\calU = \{U_i\}_{i=1}^n}

\newcommand{\rig}{\mathfrak{R}}

\newcommand{\conn}{\mathfrak{C}}

\newcommand{\distinguished}{distinguished }

 \newcommand{\atomMacro}[2]{A ^{#1}_{#2}}

 \newcommand{\odim}{\dim^{\rm open}}
 \newcommand{\cdim}{\dim^{\rm closed}}
 \newcommand{\nondegdim}{\dim^{\rm nondeg}}

\newcommand{\ncfamn}{\calD_n}

\newcommand{\R}{{\mathbb R}}

\newcommand\x{360/14}
\newcommand\y{360/6}
\newcommand\z{360/10}
\newcommand\w{360/12}

\newtheorem{theorem}{Theorem}[section]

\newtheorem{proposition}[theorem]{Proposition}
\newtheorem{corollary}[theorem]{Corollary}
\newtheorem{lemma}[theorem]{Lemma}

\theoremstyle{definition}
\newtheorem{example}[theorem]{Example}
\newtheorem{definition}[theorem]{Definition}

\newtheorem{remark}[theorem]{Remark}
\newtheorem{question}[theorem]{Question}

\newtheorem{assumption}[theorem]{Assumption}

\newcommand{\calC}{{\mathcal{C}}}
\newcommand{\calD}{{\mathcal{D}}}

\newcommand{\calU}{{\mathcal{U}}}

\newcommand{\CruzCode}{\calC_{\rm Cr}}

\usepackage[foot]{amsaddr}

\begin{document}

\title[Nondegenerate neural codes and obstructions to closed-convexity]{Nondegenerate neural codes and obstructions to \\ closed-convexity}
%
\author[P.~Chan]{Patrick Chan$^1$}
\address{$^1$Loyola University Chicago}
\author[K.~Johnston]{Katherine Johnston$^2$}
\address{$^2$Lafayette College}
\author[J.~Lent]{Joseph Lent$^3$}
\address{$^3$Pennsylvania State University}
\author[A.~Ruys de Perez]{Alexander Ruys de Perez$^4$}
\address{$^4$Texas A\&M University}
%
\author[A.~Shiu]{Anne Shiu$^4$}

\date{\today}

\maketitle
\begin{abstract}
Previous work on convexity of neural codes has produced codes that are open-convex but not closed-convex -- or vice-versa.
However, why a code is one but not the other, and how to detect such discrepancies are open questions. 
We tackle these questions in two ways. First, 
we investigate the concept of nondegeneracy introduced by Cruz 
{\em et al.}  
We extend their results to show that 
nondegeneracy precisely captures the situation when taking closures or interiors of open or closed realizations, respectively, does not change the code that is realized. 
Second, we give the first general criteria for precluding a code from being closed-convex (without ruling out open-convexity), unifying ad-hoc geometric arguments in prior works.  
One criterion is built on a phenomenon we call a rigid structure, while the other can be stated algebraically, in terms of the neural ideal of the code. 
These results complement existing criteria having the opposite purpose: precluding open-convexity but not closed-convexity.
Finally, we show that a family of codes shown by Jeffs to be not open-convex is in fact closed-convex and realizable in dimension three.
\end{abstract}

\section{Introduction} \label{sec:intro}
This work focuses on the following question: Which intersection patterns can be cut out by convex open sets in some Euclidean space, and which by convex closed sets?  
Such intersection patterns are called {\em neural codes}, 
as they arise in neuroscience as representations of the firing patterns of neurons called place cells.
A {\em place cell} 
in someone's brain
fires 
precisely when the person is in a specific region of space, called a {\em place field}~\cite{Oke2}. 
Because experimental data have shown that place fields are typically convex, 
we are interested in the question of which neural codes can be 
realized by collections of convex open or convex closed regions.  Such codes are, respectively, {\em open-convex} and {\em closed-convex}.

The main way to show that a neural code is neither open-convex nor closed-convex is to prove that it has what is called a local obstruction~\cite{new-obs, no-go, Jeffs2018convex}.  However, it is possible to have no local obstruction and yet be non-open-convex~\cite{jeffs2019sunflowers, lienkaemper2017obstructions} or non-closed-convex~\cite{cruz2019open, goldrup2020classification} or both~\cite{non-monotonicity}.
Indeed, there are neural codes that are open-convex but not closed-convex, and vice-versa; in this article, we are interested in distinguishing between these classes.  

One way to investigate open-convexity versus closed-convexity codes is to see what happens when we take closures or interiors of open or closed realizations, respectively.  We strengthen a result of Cruz {\em et al.} to show that their concept of nondegeneracy~\cite{cruz2019open} exactly captures the situation when taking closures or interiors 
yields another realization of the code (Theorem~\ref{thm:nondeg-summary}), and so such codes are both open-convex and closed-convex.

We are also interested in non-local obstructions to open- or closed-convexity.  
Criteria precluding open-convexity arise from the concept of ``wheels''~\cite{wheels} and related ideas~\cite{jeffs2019sunflowers}.  So far, however, there are no analogous criteria for ruling out closed-convexity~\cite[Remark 3.5]{non-monotonicity}.  
Accordingly, we give two such criteria, unifying ad-hoc geometric arguments in prior works~\cite{cruz2019open, goldrup2020classification}.

One criterion harnesses a phenomenon we call a rigid structure (see Definition~\ref{def:rigid} and Theorem~\ref{thm:rigid}).  The criterion is similar in spirit to criteria based on wheels, and therefore is widely applicable. 
Our results here allow us to construct the first infinite family of codes that are open-convex but not closed-convex (Theorem~\ref{thm:infinite-family-non-closed}).  
Moreover, for many codes, our criterion is straightforward to apply -- simply by inspection of a related graph (Theorem~\ref{thm:rigid-when-gph-is-cycle}).  
In general, however, this first criterion can be difficult to check.  

Our second criterion, in contrast, can be read directly from the code; it asks whether some codewords contain certain subsets of neurons (Theorem~\ref{thm:criterion-non-closed-convex-via-RF}). 
This criterion, while somewhat limited in application, nevertheless yields an algebraic signature of non-closed-convexity, arising in the neural ideal of the code (Corollary~\ref{cor:criterion-CF}).  
In this way, we mirror algebraic signatures that are known for local obstructions~\cite{curto2018algebraic}. 

Finally, 
we investigate an infinite family of codes shown by Jeffs to be not open-convex~\cite{jeffs2019sunflowers}. 
We show that these ``sunflower codes'' are closed-convex and, moreover, realizable in dimension three (Theorem~\ref{thm:sunflower-code}).  This is the first family of codes known to be closed-convex but not open-convex.

The outline of our work is as follows. 
Section~\ref{sec:background} introduces neural codes, convexity, nondegeneracy, and neural ideals.  
Section~\ref{sec:degenerate-v-nondeg} contains our results on nondegeneracy and embedding dimensions of codes.
Our criteria for precluding closed-convexity appear in Sections~\ref{sec:rigid}
and~\ref{sec:criterion-RF},
and then we present our closed-convex realization of sunflower codes in Section~\ref{sec:sunflower}.
We end with a discussion in Section~\ref{sec:discussion}.

\section{Background} \label{sec:background}
In this section, we introduce convex neural codes (Section~\ref{sec:codes-convexity}), nondegenerate realizations of neural codes (Section~\ref{sec:nondeg}), and neural ideals (Section~\ref{sec:neural-ideal}).

\subsection{Neural codes and convexity } \label{sec:codes-convexity}
For a set $Y \subseteq \mathbb{R}^d$, we let ${\rm int}(Y)$, ${\rm cl}(Y)$, and $\partial Y$ denote the interior, closure, and boundary of $Y$, respectively.  
If $\mathcal{U} = \{U_1,U_2,\dots,U_n\}$
is a 
collection of subsets of $\mathbb{R}^d$, we use the notation
${\rm int}(\calU):= \{ {\rm int}(U_1), {\rm int}(U_2), \dots,  {\rm int}(U_n)\}$ and ${\rm cl}(\calU):= \{ {\rm cl}(U_1), {\rm cl}(U_2),\dots, {\rm cl}(U_n)\}$.

Below we largely follow the notation of Cruz {\em et al.}~\cite{cruz2019open}.  

\begin{definition} \label{def:code}
A {\em neural code} $\mathcal{C}$ on $n$ neurons is a set of subsets of $[n]:= \{1,2,\dots, n\}$, that is, $\mathcal{C} \subseteq 2^{[n]}$.  Each element of $\mathcal{C}$ is a {\em codeword}.
\end{definition}

For convenience, we often use a shorthand for codewords, e.g., $124$ in place of $\{1,2,4\}$.

\begin{definition} \label{def:restricted}
Let $\mathcal{C}$ be a code on $n$ neurons, and let $\tau \subseteq [n]$.  
The {\em code obtained from $\mathcal{C}$ by restricting to $\tau$} is the neural code $\code|_{\tau} := \{\sigma \cap \tau \mid \sigma \in \mathcal{C} \}$.
\end{definition}

\begin{definition} \label{def:realize}
Let  $\mathcal{U} = \{U_1,U_2,\dots,U_n\}$
be a 
collection of 
sets in a {\em stimulus space} $X \subseteq \mathbb{R}^d$.
\begin{itemize}
    \item[(i)] For $\sigma \subseteq [n]$, the {\em atom} of $\sigma$ with respect to $\mathcal{U}$, 
    denoted by $\atomMacro{\calU}{\sigma}$,
    is the following subset of $X$:
    \begin{align*}
        \atomMacro{\calU}{\sigma} \quad := \quad 
        \left( \bigcap_{i \in \sigma} U_i \right)
        \smallsetminus  \bigcup_{j \notin \sigma} U_j~.
    \end{align*}
    By convention, the empty intersection is $\cap_{i\in \emptyset} U_i := X$.
    \item[(ii)] The neural code {\em realized} by $\calU$ and $X$, denoted by 
    ${\rm code}(\mathcal{U}, X)$, 
    is the neural code on $n$ neurons defined by:
    \begin{align*}
        \sigma \in {\rm code}(\mathcal{U}, X) ~ \Longleftrightarrow ~ 
    \atomMacro{\calU}{\sigma} \neq \emptyset~.
    \end{align*}
    In this case, $\mathcal{U}$ is a {\em realization} of $\code = {\rm code}(\mathcal{U}, X)$.
\end{itemize}
\end{definition}

\begin{definition} \label{def:convex}
A neural code $\mathcal{C}$ on $n$ neurons 
    is {\em open-convex} (or {\em closed-convex}) 
    if, for some $\mathbb{R}^d$, 
    there exist: 
    \begin{itemize}
    \item[(i)]
        a collection of 
    convex open (or, respectively, closed) sets $\fullreal$ in $\mathbb{R}^d$, 
    and 
    \item[(ii)]
     a stimulus space $X \subseteq \mathbb{R}^d$ that contains the union $\cup_{i=1}^n U_i$, 
    \end{itemize}
    such that 
    $\code=  {\rm code}(\mathcal{U}, X)$.  The minimum such value of $d$ (or $\infty$ if no such value exists) is the {\em open embedding dimension} (respectively, {\em closed embedding dimension}) of $\code$, which we denote by $\odim(\code)$ (respectively, $\cdim(\code)$).
\end{definition}

\begin{assumption} 
\label{assumption:emtpy-codeword}
Unless otherwise specified, we will assume that all codes contain the empty set as a codeword and 
that
$X=\mathbb{R}^d$, for some $d$ (see~\cite[Remark 2.19]{new-obs}).
\end{assumption}

The main way to prove that a code is neither open-convex nor closed-convex is to show that it has a ``local obstruction''~\cite{what-makes,no-go} or a generalized such obstruction~\cite{new-obs,Jeffs2018convex}.  
(Having such an obstruction, however, is not necessary to be non-open-convex or non-closed-convex~\cite{cruz2019open, lienkaemper2017obstructions}.)
These obstructions are combinatorial and topological, and we do not give a definition here.

\begin{example}[Open-convex but not closed-convex] \label{ex:goldrup-phillipson}
To our knowledge, only four neural codes in the literature have been shown to be open-convex but not closed-convex.  
The first, a code on $6$ neurons, was found by 
Cruz {\em et al.}~\cite[Lemma 2.9]{cruz2019open}: 
\[
\CruzCode ~=~ \{
    {\bf 123}, {\bf 126}, {\bf 156},  {\bf 234}, {\bf 345}, {\bf 456},
            12, 16, 23,  34,  45, 56,
            \emptyset \}~.
\]
(Following the literature, the maximal codewords are in boldface.) 
Similarly, Goldrup and Phillipson proved that the following three codes on 5 neurons are open-convex but not closed-convex~\cite[Theorem~4.1]{goldrup2020classification}: \begin{align*}
 C6 ~&=~ \{ {\bf 123}, {\bf 125}, {\bf 145}, {\bf 234},
            12, 15, 23,
            4, 
            \emptyset\}~, \\
 C10 ~&=~\{ {\bf 134}, {\bf 135},  {\bf 234}, {\bf 245},
            {\bf 12}, 13, 24, 34,
            1,2,5,
            \emptyset\}~, \\
 C15 ~&=~\{ {\bf 123}, {\bf 125},
            {\bf 145} ,{\bf  234}, {\bf 345},
            12,15,23,34,45, \emptyset\}~.
\end{align*}
One aim of this work is to give general criteria that in particular show that  
the codes $\CruzCode$, $C6$, $C10$, and $C15$
are non-closed-convex.  See Proposition~\ref{prop:examples-rigid} and Example~\ref{ex:3-codes-via-RF-criterion}.  
\end{example}

\begin{example}[Closed-convex but not open-convex] \label{ex:closed-not-open}
The following is the first code that was shown to be closed-convex~\cite{cruz2019open}
but not open-convex (it is also the first code known to lack local obstructions and yet be non-open-convex)~\cite[Theorem 3.1]{lienkaemper2017obstructions}:
\[
\code^{\star} = \{ {\bf 2345}, {\bf 123}, {\bf 134}, {\bf 145}, 13, 14, 23, 34, 45, 3, 4, \emptyset \}~.
\]
An infinite family of codes that are closed-convex but not open-convex appears later in Theorem~\ref{thm:sunflower-code}.
\end{example}

As mentioned above, codes with local obstructions are neither open-convex nor closed-convex.
Here, however, we are interested in codes without local obstructions.  

\begin{example}[Neither open-convex nor closed-convex] \label{ex:non-monotone-code}
The following code on 8 neurons is not closed-convex, despite having no local obstructions~\cite[Theorem 3.2]{non-monotonicity}:
\begin{align*} 
\mathcal C_8 = 
    \{
     \mathbf{12378}, 
    \mathbf{1457}, \mathbf{2456}, \mathbf{3468},
    278, 
    17, 38, 45, 46, 
    2, 
    \emptyset
    \}    ~.
\end{align*}
This code is also not open-convex.  Indeed, the code obtained from $\code_8$ by restricting to $\{1,2,3,4,5,6\}$ is (up to permuting neurons) the minimally non-open-convex code in~\cite[Theorem 5.10]{jeffs2020morphisms}, and restriction preserves convexity.  
Another example of a non-open-convex and non-closed-convex code with no local obstructions is given in~\cite[Theorem 3.11]{non-monotonicity}.
\end{example}

Examples~\ref{ex:goldrup-phillipson} and~\ref{ex:closed-not-open} above feature codes with finite open embedding dimension, but infinite closed embedding dimension  -- or vice-versa.  
Next, we see that
even when both embedding dimensions are finite, they can differ.


\begin{figure}[ht] 
    \centering
    \begin{tikzpicture}[scale=0.35]
    \draw[fill=blue, opacity=0.4] (0,7)--(-9,-6)--(-5,-6)--(0,7);
    \draw[fill=red, opacity=0.2] (0,7)--(-2,-6)--(2,-6)--(0,7);
    \draw[fill=violet, opacity=0.1] (0,7)--(9,-6)--(5,-6)--(0,7);
    \draw[fill=black, opacity=0.2] (-9,-6)--(9,-6)--(6.9,-3)--(-6.9,-3);
    \draw[very thick, line cap = round] (0,7)--(-9,-6)--(-5,-6)--(0,7);
    \draw[very thick, line cap = round] (0,7)--(-2,-6)--(2,-6)--(0,7);
    \draw[very thick, line cap = round] (0,7)--(9,-6)--(5,-6)--(0,7);
    \draw[very thick] (-9,-6)--(9,-6)--(6.9,-3)--(-6.9,-3)--cycle;
    \node[] at (7,2) {\Large{$\emptyset$}};
    \node[] at (-6,-4.5) {\Large{14}};
    \node[] at (-3,-4.5) {\Large{4}};
    \node[] at (0,-4.5) {\Large{24}};
    \node[] at (0, -7) {{\color{red} $U_2$ }};
    \node[] at (3,-4.5) {\Large{4}};
    \node[] at (6,-4.5) {\Large{34}};
    \node[] at (9.5, -4.5) {{\color{black} $U_4$ }};
    \node[] at (-4, 0) {\Large{1}};
    \node[] at (-6, 0) {{\color{blue} $U_1$ }};
    \node[] at (0,0) {\Large{2}};
    \node[] at (4,0) {\Large{3}};
    \node[] at (6, 0) {{\color{violet} $U_3$ }};
    \node[] at (1.6, 7) {\Large{123}};
    \end{tikzpicture}
    \caption{A closed-convex realization of the code in~\eqref{eq:code-example} ~(cf.~\cite[Figure~2]{non-monotonicity}).}
    \label{fig:open-vs-closed-dim}
 \end{figure}
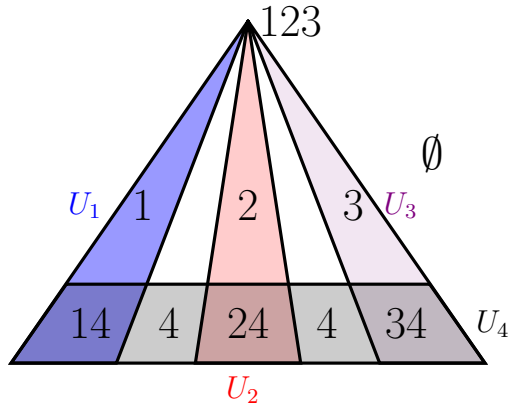

\begin{example}[Non-equality of open and closed embedding dimensions] \label{ex:open-dim-not-equal-closed-dim}
The following code is both open-convex and closed-convex:
\begin{align} \label{eq:code-example}
\mathcal{C}_\theta ~:=~
\{ 
{\bf 123}, {\bf 14}, {\bf 24}, {\bf 34}, 1,2,3,4, \emptyset
\}~.
\end{align}
Indeed, a closed-convex realization in $\mathbb{R}^2$ is shown in Figure~\ref{fig:open-vs-closed-dim}, 
and an open-convex realization in $\mathbb{R}^3$ 
appears 
in~\cite[Example~5.5]{jeffs2019embedding}. 
We claim that such convex realizations do not exist in lower dimensions, that is, 
$\cdim(\mathcal{C}_\theta)=2$ and $\odim(\mathcal{C}_\theta)=3$.  First, the open embedding dimension was shown in~\cite[Table 2]{what-makes}.  
Next, it is straightforward to check that $\mathcal{C}_\theta$ does not have a closed-convex realization in $\mathbb{R}$ (see the results in~\cite{zvi-yan}).
\end{example}

\begin{remark} \label{rem:differ-emb-dim}
The open and closed embedding dimensions can differ by any amount.  
In Section~\ref{sec:sunflower}, will see ``sunflower codes'' $S_n$ 
with closed embedding dimension at most $3$ (Theorem~\ref{thm:sunflower-code}), but open embedding dimension equal to $n$~\cite[\S 5]{jeffs2019embedding}.  As for codes with larger 
closed embedding dimension than open, 
 we posed this question in an earlier version of this work, and, notably, 
Jeffs has now constructed such codes~\cite{open-closed-nondeg}.

\end{remark}

\subsection{Nondegenerate realizations of neural codes} \label{sec:nondeg}
We recall the following definition introduced by Cruz {\em et al.}~\cite{cruz2019open}.

\begin{definition}    \label{def:nondegen}
A collection of subsets 
$\fullreal$ of $\mathbb{R}^d$ is \emph{nondegenerate} if:
        \begin{enumerate}[label=(\roman*)]
            \item 
            for every nonempty open set 
            $S_o \subseteq \mathbb{R}^d$ 
            and every nonempty atom
            $A^\mathcal{U}_{\sigma}$,
            if the intersection
            $A^\mathcal{U}_{\sigma} \cap S_o$
            is nonempty, then the interior is also nonempty: ${\rm int}(A^\mathcal{U}_{\sigma}\cap S_o) \neq \emptyset$; and
            \item for all nonempty $\sigma \subseteq [n]$,
            the following containment holds: $\left(\bigcap_{i \in \sigma} \partial U_i\right) \subseteq \partial \left( \bigcap_{i \in \sigma}U_i \right)$.
        \end{enumerate}
\end{definition}

The following result is due to Cruz {\em et al.}~\cite[Lemma 2.11]{cruz2019open}.      \begin{lemma}        \label{cruzlemmadegen}
Let $\fullreal$ be a collection of convex sets in $\mathbb{R}^d$.  Then:
        \begin{enumerate}[label=(\roman*)]
            \item if every $U_i$ is open and $\mathcal{U}$ satisfies Definition~\ref{def:nondegen}$(ii)$, then $\calU$ also satisfies Definition~\ref{def:nondegen}$(i)$.
            \item if every $U_i$ is closed and $\mathcal{U}$ satisfies Definition \ref{def:nondegen}$(i)$, then $\calU$ also satisfies Definition \ref{def:nondegen}$(ii)$.
        \end{enumerate}
    \end{lemma}

\begin{example} \label{ex:deg-vs-nondeg}
For the code $\code =\{ {\bf 1} ,{\bf 2}, \emptyset \}$, the open-convex realization in $\mathbb{R}$ given by the intervals $U_1=(-1,0)$ and $U_2=(0,1)$ is degenerate, but $U_1=(-1,0)$ and $U_2=(1,2)$ define a realization that is nondegenerate.  Similarly, for $\code = \{ {\bf 12}, 1,2, \emptyset \}$, the closed-convex realization of $U_1= [-1,0]$ and $U_2= [0,1]$ is degenerate, while $U_1= [-1,1]$ and $U_2=[0,2]$ define a realization that is nondegenerate. Finally, the closed-convex realization shown in Figure~\ref{fig:open-vs-closed-dim} is degenerate, as the atom of the codeword $123$ has empty interior.
\end{example}

The following result is also due to Cruz {\em et al.}~\cite[Theorem 2.12]{cruz2019open}.
   \begin{proposition}
    \label{prop:Cruz}
    Let $\fullreal$ be a nondegenerate collection of convex sets in $\mathbb{R}^d$.  Then:
        \begin{enumerate}[label=(\roman*)]
        \item  if every $U_i$ is open, then 
        $
        {\rm code}
        (\mathcal{U},\mathbb{R}^d) = {\rm code}({\rm cl}(\mathcal{U}),\mathbb{R}^d)$; and
        \item if
        every $U_i$ is closed, then $ {\rm code}(\mathcal{U},\mathbb{R}^d) = {\rm code}({\rm int} (\mathcal{U}),\mathbb{R}^d)$.
        \end{enumerate}
   \end{proposition} 

In the next section,
we prove a converse to Proposition~\ref{prop:Cruz} (see Theorem~\ref{thm:nondeg-summary}).

\begin{proposition}    \label{prop:OndvCnd}
Let $\code$ be a neural code.  
Let $d_o$ (respectively, $d_c$) be the minimum value (or~$\infty$ if no such value exists) such that
$\code = {\rm code}(\mathcal{U}, \mathbb{R}^{d_o})$ 
(respectively, $\code = {\rm code}(\mathcal{U}, \mathbb{R}^{d_c})$) for
some \uline{nondegenerate} collection of convex open (respectively, closed) sets $\mathcal U$ in $\mathbb{R}^{d_o}$ (respectively, $\mathbb{R}^{d_c}$).
Then $d_o=d_c$.
\end{proposition}
    
    \begin{proof}
    The first implication of Proposition~\ref{prop:Cruz} implies the inequality $d_o \geq d_c$, and the second implication implies the opposite inequality, $d_0 \leq d_c$. 
    \end{proof}

In light of Proposition~\ref{prop:OndvCnd}, we introduce the following definition, which captures the minimal dimension in which a code has a nondegenerate open-convex or closed-convex realization.

\begin{definition} \label{def:nondeg-dim}
The {\em nondegenerate embedding dimension} of a neural code $\code$, denoted by $\nondegdim(\code)$, is the value of $d_o=d_c$ in Proposition~\ref{prop:OndvCnd}.
\end{definition}

\begin{proposition}
    \label{prop:dvnd}
For every neural code $\code$,
the following inequalities hold:
\begin{align}\label{eq:dim-inequalities}
    \odim(\code) \leq 
    \nondegdim(\code)  
    \quad \quad {\rm and} \quad \quad
    \cdim(\code) \leq 
    \nondegdim(\code)~.
\end{align}
\end{proposition}

\begin{proof}
Every nondegenerate, open-convex (respectively, closed-convex) realization of $\code$ is also an open-convex (respectively, closed-convex) realization of $\code$. 
\end{proof}

We will investigate when the inequalities~\eqref{eq:dim-inequalities} are equalities (for instance, see Theorem~\ref{thm:codes-on-up-to-3}).  The following examples show that these inequalities are not equalities in general.

\begin{example} \label{ex:nondeg-dim-not-equal}
We saw that the code 
$\code_\theta$
in Example~\ref{ex:open-dim-not-equal-closed-dim} satisfies 
$\odim(\mathcal{C}_\theta) = 3$ and $ \cdim(\mathcal{C}_\theta) = 2$.  
We claim that $\nondegdim(\mathcal{C}_\theta) =3$.  
Indeed, viewing the $3$-dimensional realization of $\code_\theta$ in~\cite[Example~5.5]{jeffs2019embedding} as a closed realization, it is easy to check that Definition~\ref{def:nondegen}(i) holds, and so by Lemma~\ref{cruzlemmadegen}(ii), the realization is nondegenerate.  Now Proposition~\ref{prop:dvnd} implies that
 $\nondegdim(\mathcal{C}_\theta) =3$.
\end{example}

Other codes with $\cdim(\code) <    \nondegdim(\code)$ are found in~\cite[\S6]{jeffs2019embedding}
and~\cite[Lemma 3.3]{non-monotonicity}.

Next, we consider some related questions.
\begin{question} \label{q:open-vs-closed-vs-nondeg} ~
\begin{enumerate}
    \item Is there a code $\code$ with $\odim(\code) <    \nondegdim(\code) < \infty$?
    \item Is there a code $\code$ with $\odim(\code) < \infty$ and $\cdim(\code) < \infty$, but $\nondegdim(\code) = \infty$?
\end{enumerate}
\end{question}

We posed these questions in an earlier version of this work, and recently Jeffs resolved both
questions affirmatively~\cite{open-closed-nondeg}.  Notably, his constructions rely on our theory of rigid structures, which appears later in this article (Section~\ref{sec:rigid}).

%

In Question~\ref{q:open-vs-closed-vs-nondeg}(1), if we remove the requirement that the nondegenerate embedding dimension is 
finite, then we can find several such codes, as in the following example.

\begin{example} \label{ex:infinite-deg-dim}
For codes $\code$ that are open-convex but not closed-convex (such as those in Example~\ref{ex:goldrup-phillipson})
, $\odim(\code) < \infty$ while $\nondegdim(\code)=\infty$.
Similarly, for codes that are closed-convex but not open-convex (as in Example~\ref{ex:closed-not-open}),  
$\cdim(\code) < \infty$ but $\nondegdim(\code)=\infty$.
\end{example}




\subsection{The neural ideal and its canonical form} \label{sec:neural-ideal}
In this subsection, we introduce neural ideals, which capture all the information in a neural code.  
Neural ideals have been harnessed to study convexity and other properties of neural codes~\cite{ curto2018algebraic, curto2013neural, GB, neural-ideal-homom, neural-ideal-sage,factor-complex}.

A {\em pseudo-monomial} in $\mathbb{F}_2[x_1, x_2, \dots, x_n]$ is a polynomial of the form 
\[
f~=~ \prod_{i\in \sigma} x_i\prod_{j\in\tau}(1+x_j)~,
\]
where $\sigma,\tau\subseteq[n]$ with $\sigma\cap\tau=\emptyset$.  
Each $v\in \{0,1\}^n$ defines a pseudo-monomial $\rho_v$, which is the {\em characteristic function} for $v$ (i.e., 
$\rho_v(x)=1$ if and only if $x=v$):
\begin{align*}
\rho_v~:=~\prod_{i=1}^n(1-v_i-x_i)=\prod_{\{i \mid v_i=1\}}x_i\prod_{\{j \mid v_j=0\}}(1+x_j)~.
\end{align*}

For a codeword $\sigma \subseteq [n]$ (e.g., $\sigma=134$ with $n=4$), we let $v(\sigma) \in \{0,1\}^n$ denote the corresponding $0/1$ vector (e.g., $v(\sigma)=(1,0,1,1)$).  That is,  $v(\sigma)_i=1$ if and only if $i \in \sigma$.

\begin{definition} \label{def:neural-ideal}
Let $\code$ be a neural code on $n$ neurons. The {\em neural ideal} $J_{\code}$ of $\code$ is the ideal in $\mathbb{F}_2[x_1, x_2, \dots, x_n]$ generated by all characteristic functions
$\rho_v$ for $v\not\in \code $:
	\begin{align*}
    J_{\code}~:=~\langle \{\rho_v \mid v\not\in \code \}\rangle~.
    \end{align*}
\end{definition}

A pseudo-monomial $f$ in an ideal $J$ in $\mathbb{F}_2[x_1, x_2, \dots,x_n]$ is {\em minimal} if there does not exist a pseudo-monomial $g\in J$, where $g \neq f$,
such that $f=gh$ for some $h\in \mathbb{F}_2[x_1, x_2, \dots,x_n]$.

\begin{definition} \label{def:CF}
The {\em canonical form} of a neural ideal $J_{\code}$, which we denote by 
${\rm CF}(J_{\code})$, 
is the set of all minimal pseudo-monomials of $J_{\code}$. 
\end{definition}

The canonical form ${\rm CF}(J_{\code})$ is a generating set for the neural ideal $J_{\code}$~\cite{curto2013neural}, and there is software for efficiently computing $J_{\code}$ and ${\rm CF}(J_{\code})$~\cite{neural-ideal-sage}.  

\begin{example} \label{ex:canonical-form}
The canonical forms for the codes 
in Example~\ref{ex:goldrup-phillipson} are as follows: 
\begin{align*}
%
{\rm CF}(J_{ C6 }) ~=~
&\{(x_1+1)x_5,~(x_2+1)x_3,~x_3x_5,~(x_1+1)x_2(x_3+1),~x_1x_2x_4,~x_2(x_3+1)x_4,~\\
&x_2x_4x_5,~x_1(x_2+1)(x_5+1),~x_1x_4(x_5+1),~x_1x_3x_4\}~, \\
%
{\rm CF}(J_{ C10 }) ~=~
&\{(x_2+1)(x_3+1)x_4,~(x_1+1)x_3(x_4+1),~x_1x_4x_5,~(x_2+1)x_4x_5,~\\
&(x_1+1)x_3x_5,~x_2x_3x_5,~x_3x_4x_5,~x_1(x_3+1)x_4,~x_1(x_3+1)x_5,~\\
&x_1x_2x_3,~x_1x_2x_4,~x_1x_2x_5,~x_2x_3(x_4+1),~x_2(x_4+1)x_5\}~, \\
%
%
{\rm CF}(J_{ C15 }) ~=~
&\{x_1(x_2+1)(x_5+1),~x_1x_4(x_5+1),~(x_3+1)x_4(x_5+1),~(x_1+1)x_2(x_3+1),~\\
&x_1x_2x_4,~x_2(x_3+1)x_4,~x_2x_4x_5,~x_1x_3x_4,~(x_1+1)x_2x_5,~(x_1+1)(x_4+1)x_5,~\\
&x_1(x_2+1)x_3,~(x_2+1)x_3(x_4+1),~x_1x_3x_5,~x_2x_3x_5,~x_3(x_4+1)x_5\}~.
\end{align*}

\end{example}

Pseudo-monomials in $J_{\code}$ (and thus in the canonical form) 
can be translated into relationships among receptive fields. 
To state this result, we need the following notation for $\sigma\subseteq [n]$:
\[
x_\sigma ~:=~\prod_{i\in\sigma}x_i \quad 
\text{ and } \quad
U_\sigma ~:=~ \bigcap_{i\in\sigma} U_i~.
\]
(As mentioned earlier, recall that the empty intersection is the full space $X$.) 
The following result is due to Curto {\em et al.}~\cite[Lemma~4.2]{curto2013neural}.
\begin{lemma} \label{lem:receptivefields}
	Let $X$ be a stimulus space, let $\mathcal{U}=\{U_i\}_{i=1}^n$ be a collection of  
    subsets of $X$, and consider the code $\code = {\rm code}(\mathcal{U}, X)$. 
    Then for every pair of subsets $\sigma,\tau\subseteq [n]$, 
    $$x_\sigma\prod_{i\in\tau}(1+x_i)~\in~ J_{\code} \quad \iff \quad  U_\sigma ~\subseteq~ \bigcup_{i\in\tau}U_i~.$$
    In particular, $x_\sigma \in J_{\code} \iff U_\sigma = \emptyset $.	
\end{lemma}
The relations $U_\sigma = \emptyset$ 
and 
$U_\sigma\subseteq\cup_{i\in\tau}U_i$, 
as long as $\sigma \cap \tau =  \emptyset$
and $U_{\sigma} \cap U_i \neq \emptyset$ for all $i \in \tau$, 
are called {\em receptive-field relationships} (RF relationships). There are more types of RF relationships~\cite{curto2013neural, GB, morvant}, but here we do not need them.
We will use Lemma~\ref{lem:receptivefields} 
in Section~\ref{sec:criterion-RF} to prove a criterion that precludes closed-convexity.

\section{Degenerate and nondegenerate codes} \label{sec:degenerate-v-nondeg}
The main result of this section, Theorem~\ref{thm:nondeg-summary}, shows that the concept of nondegeneracy exactly captures when the operations of taking interiors or closures yield the same code, thereby clarifying (as Cruz {\em et al.} articulated~\cite{cruz2019open}) that nondegeneracy is the ``correct'' setting for studying convexity of neural codes.

\begin{theorem} \label{thm:nondeg-summary}
Let $\fullreal$ be a collection of convex sets in $\mathbb{R}^d$.  
    \begin{enumerate}[label=(\roman*)]
        \item  Assume every $U_i$ is open. 
        Then $\calU$ is nondegenerate if and only if  
        ${\rm code}(\mathcal{U},\mathbb{R}^d) = {\rm code}({\rm cl}(\mathcal{U}),\mathbb{R}^d)$.
        \item Assume every $U_i$ is closed. 
        Then $\calU$ is nondegenerate if and only if
        $ {\rm code}(\mathcal{U},\mathbb{R}^d) = {\rm code}({\rm int} (\mathcal{U}),\mathbb{R}^d)$.
        \end{enumerate}    
\end{theorem}

Theorem~\ref{thm:nondeg-summary} follows directly from Proposition~\ref{prop:Cruz} (due to Cruz {\em et al.}~\cite{cruz2019open}) and Proposition~\ref{prop:degen-cover}. 

    \begin{proposition}  \label{prop:degen-cover}
 Let $\fullreal$ be a degenerate collection of convex sets in $\mathbb{R}^d$.  Then:
    \begin{enumerate}[label=(\roman*)]
        \item  if every $U_i$ is open, then 
        ${\rm code}(\mathcal{U},\mathbb{R}^d) \neq {\rm code}({\rm cl}(\mathcal{U}),\mathbb{R}^d)$; and
        \item if
        every $U_i$ is closed, then $ {\rm code}(\mathcal{U},\mathbb{R}^d) \neq {\rm code}({\rm int} (\mathcal{U}),\mathbb{R}^d)$.
        \end{enumerate}    
    \end{proposition}
    
    \begin{proof}
    $(i)$: Assume that $\mathcal{U}$ is degenerate and every $U_i$ is open and convex. By Lemma \ref{cruzlemmadegen}, there exists some nonempty $\sigma \subseteq [n]$ 
    such that 
    \begin{align} \label{eq:subsetneq}
    \bigcap_{i \in \sigma} \partial U_i 
    \quad \nsubseteq \quad  \partial \left( \bigcap_{i \in \sigma}U_i \right)~.
    \end{align}

    We claim that $\cap_{i \in \sigma} U_i = \emptyset$.  To prove our claim, assume for contradiction that $\cap_{i \in \sigma} U_i \neq \emptyset$ and so there exists some $y \in \cap_{i \in \sigma} U_i$.  
         (See Figure~\ref{fig:proof-x-y-z} for the case of $|\sigma|=2$.)
    We will show the containment  
    $\cap_{i \in \sigma} \partial U_i  
     \subseteq  \partial \left( \cap_{i \in \sigma}U_i \right)$, contradicting~\eqref{eq:subsetneq}.  
     Accordingly, let $x \in \cap_{i \in \sigma} \partial U_i$.  
     By construction, $x \neq y$, because $x$ is on the boundary of the open set $U_i$ (for any $i \in \sigma$) and so is not in $U_i$ (whereas $y \in U_i$).   
     Consider the line segment $L$ from $x$ (which is in $\partial U_i$ for all $i \in \sigma$) to $y$ (which is in $U_i$ for all $i \in \sigma$).  
    Let $\widetilde{L}$ denote the line that is the affine span of $L$ (so, $L \subseteq \widetilde{L}$).  Then, for $i\in \sigma$, the intersection $U_i \cap \widetilde{L}$, which contains $y$, is an open subinterval of $\widetilde{L}$ (because $U_i$ is open and convex), and it is straightforward to check that $x$ is an endpoint of this subinterval.  We conclude that $L\smallsetminus \{x\}\subseteq U_i$ (for all $i \in \sigma$).

     Now, take an open neighborhood $B$ (in $\mathbb{R}^d$) of~$x$.  Fix $i_0 \in \sigma$.  As $x \in \partial U_{i_0}$, there exists a point $z \in B \smallsetminus U_{i_0}$ and so $z  \in B \smallsetminus \left( \cap_{i\in \sigma} U_{i}\right)$.  On the other hand, 
    every point $z'$ on the line segment $L$ that is sufficiently close to (but not equal to) $x$ 
    is in $B$ and also in $U_i$ for all $i \in \sigma$.  Hence, $B$ contains a point $z'$ that is in $\cap_{i \in \sigma}U_i$ and also a point $z$ that is not (see Figure~\ref{fig:proof-x-y-z}).  We conclude that $x \in \partial \left( \cap_{i \in \sigma}U_i \right)$, and so we have reached a contradiction.  Our claim therefore is true.
    
    Our claim implies that {\em no} codeword of ${\rm code}(\mathcal{U},\mathbb{R}^d) $ contains $\sigma$.  
    On the other hand, the containment~\eqref{eq:subsetneq} implies 
    that $ \bigcap_{i \in \sigma} \partial U_i$ is nonempty, and hence $ \bigcap_{i \in \sigma} {\rm cl}( U_i)$ is also nonempty.  So, some codeword of ${\rm code}({\rm cl}(\mathcal{U}), \mathbb{R}^d)$  contains $\sigma$, and thus ${\rm code}(\mathcal{U},\mathbb{R}^d) \neq {\rm code}({\rm cl}(\mathcal{U}),\mathbb{R}^d)$.
    

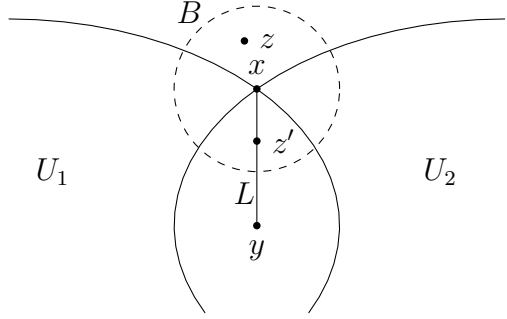
\begin{figure}
    \begin{center}
    \begin{tikzpicture} [scale=1.1]
	    \tikzstyle{vertex}=[circle,fill=black, inner sep = 1pt]
	        \draw [domain=-25:90] plot ({4*cos(\x)}, {2.5*sin(\x)});
            \draw [domain=90:205] plot ({4*cos(\x)+6}, {2.5*sin(\x)});
            \node [vertex, label=$x$] at (3,{2.5*sin(acos(3/4))}) {};
            \node [vertex, label=below:$y$] at (3,0) {};
            \draw (3,0) -- (3,{2.5*sin(acos(3/4))});
            \node [label=$L$] at (2.85,0) {};
            \draw [dashed] (3,{2.5*sin(acos(3/4))}) circle (1);
            \node [label=left:$B$] at (2.6,{3*sin(acos(3/4))+.6}) {};
            \node [vertex, label=right:$z$] at (2.85,{3*sin(acos(3/4))+.25}) {};
            \node [vertex, label=right:$z'$] at (3,{2*sin(acos(3/4))-.3}) {};
            \node [label=left:$U_{1}$] at (1,.65) {};
            \node [label=right:$U_{2}$] at (4.75,.65) {};
    \end{tikzpicture}
    \end{center}
     \centering
    \caption{Ideas in the proof of Proposition~\ref{prop:degen-cover}(i), when $\sigma = \{1,2\}$.}
    \label{fig:proof-x-y-z}
    \end{figure}

    \noindent $(ii)$: Assume that $\mathcal{U}$ is  degenerate and every $U_i$ is closed and convex. By Lemma \ref{cruzlemmadegen}, there exist a nonempty open set $S_o \subseteq \mathbb{R}^d$ and an atom $A^\mathcal{U}_{c}$ 
    such that 
    the following set 
    is nonempty but has empty interior:
    \begin{align} \label{eq:intersection-atom}
    A^\mathcal{U}_{c} \cap S_o \quad =\quad  
            \left( \bigcap_{i \in c} U_i \right) 
            \cap 
            \left( 
            \bigcap_{j \notin c} (\mathbb{R}^d \smallsetminus U_j) \right) 
            \cap S_o
            \quad =\quad  
            \left( \bigcap_{i \in c} U_i \right) 
            \cap
            Z
            ~,
    \end{align}
where 
$Z:=\left(             
            \bigcap_{j \notin c} (\mathbb{R}^d \smallsetminus U_j) \right) 
            \cap S_o$.  
            This set $Z$ is open (as it is the intersection of open sets). 
            Therefore, taking the interiors of the sets in~\eqref{eq:intersection-atom} (and recalling that taking interiors and intersections commute, see e.g.~\cite[Lemma A.1]{cruz2019open}) yields:
    \begin{align} \label{eq:intersection-atom-interior}
   \emptyset \quad = \quad
    {\rm int} \left( A^\mathcal{U}_{c} \cap S_o \right) 
    \quad =\quad  
            \left( \bigcap_{i \in c} {\rm int}(U_i) \right)
            \cap
            {\rm int} (Z)
    \quad =\quad  
        \left( \bigcap_{i \in c} {\rm int}(U_i) \right)
            \cap
            Z
            ~.
    \end{align}
We claim that the
intersection 
$  \bigcap_{i \in c} {\rm int}(U_i) $ is empty.
We prove this claim by contradiction.  Accordingly, assume there exists 
$y \in \bigcap_{i \in c} {\rm int}(U_i) $.  
As the set in~\eqref{eq:intersection-atom} is nonempty, there exists 
$x \in  \left( \bigcap_{i \in c} U_i \right) 
            \cap
            Z$.  Let $L$ denote the line segment from $x$ to $y$ (the $|c|=2$ case looks much like what is depicted in Figure~\ref{fig:proof-x-y-z}). 
For $i \in c$, we have the containment $(L \smallsetminus \{x\}) \subseteq {\rm int} (U_i)$, because $x\in U_i$, $y \in {\rm int} (U_i)$, and $U_i$ is convex.  Next, $x$ is in the open set $Z$, so all points $z'$ on $L$ that are sufficiently close to, but not equal to, $x$ are also in $Z$.  Hence, all such points $z'$ are in 
$        \left( \bigcap_{i \in c} {\rm int}(U_i) \right)
            \cap
            Z$, which contradicts~\eqref{eq:intersection-atom-interior}.  So, our claim is true.

We conclude the following: 
    \begin{align*}
    A^{{\rm int}(\mathcal{U})}_{c}
    \quad =\quad  
            \left( \bigcap_{i \in c} {\rm int}(U_i) \right) 
            \smallsetminus 
            \left( \bigcup_{j \notin c} {\rm int}(U_j) \right)
        \quad \subseteq \quad  
        \bigcap_{i \in \sigma} {\rm int}(U_i) 
        \quad = \quad         \emptyset~.
    \end{align*}
Thus, $c$ is {\em not} a codeword of ${\rm code}( {\rm int}(\mathcal{U}),\mathbb{R}^d) $, but is a codeword of 
${\rm code}(\mathcal{U},\mathbb{R}^d) $. 
    \end{proof}

We end this section by 
proving one case 
when the inequalities~\eqref{eq:dim-inequalities} on embedding dimensions from Proposition~\ref{prop:dvnd} are in fact equalities -- and then  discussing  
two more such cases. 
\begin{theorem} \label{thm:codes-on-up-to-3}
Let $\code$ be a neural code on $n$ neurons.  If $n \leq 3$, then all embedding dimensions are equal:
\begin{align} \label{eq:dim-equality-3}
    \odim(\code) ~=~ 
    \nondegdim(\code)  
    ~=~
    \cdim(\code) ~.
\end{align}
\end{theorem}

\begin{proof}
For $n \leq 2$, this result is easy to check ($\nondegdim(\code) = 1$ for all such codes).  

Assume $n=3$.  By Proposition~\ref{prop:dvnd}, it suffices to 
prove the desired equalities~\eqref{eq:dim-equality-3} for codes~$\code$ that are open-convex or closed-convex.

First 
consider the case when 
$\code$ is open-convex.  
A list of these codes and their open embedding dimensions can be read off (through their simplicial complex) from~\cite[Table~2]{what-makes}.  
For those codes with $\odim(\code)=1$, the equalities~\eqref{eq:dim-equality-3} are easy to check.  
Now consider the case when $\odim(\code)=2$.  An open-convex realization $\calU$ (in $\mathbb{R}^2$) of $\code$ is shown in~\cite[Figure~7]{what-makes}, and ${\rm cl}(\calU)$ is easily seen to also realize $\code$.  We conclude from Theorem~\ref{thm:nondeg-summary} that $\calU$ is nondegenerate, and so the equalities~\eqref{eq:dim-equality-3} hold.

Now 
consider the remaining case: 
$\code$ is closed-convex.  Then $\code$ has no local obstructions~\cite[Proposition 2.6]{cruz2019open} and so (as $n \leq 3$) is open-convex~\cite{what-makes}.  So, by the prior case, we are done.
\end{proof}

Theorem~\ref{thm:codes-on-up-to-3}
does not extend to codes on $n=4$ neurons (recall the code in Examples~\ref{ex:open-dim-not-equal-closed-dim} and~\ref{ex:nondeg-dim-not-equal}).

\begin{remark} \label{rem:dim-conjectures-now-resolved}
In a prior version of this article, we posed the following two conjectures, which, if true, would yield more cases in which Theorem~\ref{thm:codes-on-up-to-3} could be extended:
\begin{enumerate}
    \item For a code $\code$, if $
    \odim(\code) =1$ or $
    \cdim(\code) =1$, 
    then     
   $
    \nondegdim(\code) =1 
    $.
    \item For a code $\code$, if $
    \odim(\code) =
    \cdim(\code)$, 
    then     
    $\odim(\code) =
    \nondegdim(\code)  
    $.
\end{enumerate}
Both conjectures have been resolved recently by Jeffs~\cite{open-closed-nondeg}.  The first conjecture is true. The second, however, is not.  Interestingly, Jeffs's constructions for this second result rely in a crucial way on our theory of rigid structures, which is the topic we turn to next.  
\end{remark}

%
%

\section{Precluding closed-convexity using rigid structures} \label{sec:rigid}

In this section, we elucidate
the geometric mechanisms that underlie the known cases of neural codes that are open-convex but non-closed-convex.  Specifically, we
give a criterion for non-closed-convexity in terms of ``rigid structures'' (Theorem~\ref{thm:rigid}), 
a special case of which can be read directly from the code (Theorem~\ref{thm:rigid-when-gph-is-cycle}), 
and then show that these results apply to the relevant codes we saw earlier (Proposition~\ref{prop:examples-rigid}).  

As this section is somewhat long, we give the reader a more detailed outline.  
In Section~\ref{sec:rigid-structure}, we define rigid structures and codeword-containment graphs (Definitions~\ref{def:rigid} and~\ref{def:order-forcing-graph}) and use these concepts to preclude closed-convexity.  Next, in Section~\ref{sec:path}, we give sufficient conditions for rigidity when the codeword-containment graph is a path (Proposition~\ref{prop:path-rigid}), and then, 
in Section~\ref{sec:proof-cycle}, we use that result to prove Theorem~\ref{thm:rigid-when-gph-is-cycle}.  Finally, we illustrate our results through many examples in Section~\ref{sec:examples-rigid}.

\subsection{Rigid structures} \label{sec:rigid-structure}
We begin by motivating the idea behind rigid structures through the code 
$C6 = \{ {\bf 123}, {\bf 125}, {\bf 145}, {\bf 234},
            12, 15, 23,
            4, 
            \emptyset\}$
from Example~\ref{ex:goldrup-phillipson}.  In Figure~\ref{fig:C6-rigid}, we show the nonempty codewords divided into two pieces: the top can be viewed as the restriction of the code to the set of neurons $\{1,2,3,5\}$, and the bottom is $U_4$.  Also, each piece (as we will see below) is ``rigid'' in the sense that the union (of the corresponding atoms) is convex in every closed-convex realization (of a restricted code).  
The reason for non-closed-convexity is now apparent:
in order for the ends of the two rigid structures to fit together -- the atom for $145$ is contained in the region indicated by $15$ and similarly for $234$ and $23$ -- one of the structures must ``bend'' and therefore is not convex.

    \begin{figure}[ht]
        \centering
        \begin{tikzpicture}[thick,scale=.5, every node/.style={scale=0.8}]
            \begin{scope}
             	\node[label={$C6 | _ {1235}$}] at (21.5,0.2) {};
                \draw[pattern=north west lines] (0,0) rectangle (20,2){};
                \draw (4,0)--(4,2) {};
                \draw (8,0)--(8,2) {};
                \draw (12,0)--(12,2) {};
                \draw (16,0)--(16,2) {};
                \draw(2,1) node [fill=white,draw,rounded corners]{$15$};
                \draw(6,1) node [fill=white,draw,rounded corners]{$125$};
                \draw(10,1) node [fill=white,draw,rounded corners]{$12$};
                \draw(14,1) node [fill=white,draw,rounded corners]{$123$};
                \draw(18,1) node [fill=white,draw,rounded corners]{$23$};
                
             	\node[label={$U_4$}] at (17,-2.2) {};
                \draw[pattern= dots] (4,-1) rectangle (16,-2){};
                \draw(5.5,-1.5) node [fill=white,draw,rounded corners]{$145$};
                \draw(10,-1.5) node [fill=white,draw,rounded corners]{$4$};
                \draw(14.5,-1.5) node [fill=white,draw,rounded corners]{$234$};
                \draw(7,-1)--(7,-2){};
                \draw(13,-1)--(13,-2){};
                
                \draw[dotted] (2,0) -- (5.5,-1);
                \draw[dotted] (18,0) -- (14.5, -1);
            \end{scope}
        \end{tikzpicture}
        \caption{The code $C6$, viewed as the union of two rigid structures. Dotted lines indicate pairs of codewords arising from distinct rigid structures, with one contained in the other.}
                \label{fig:C6-rigid}
    \end{figure}
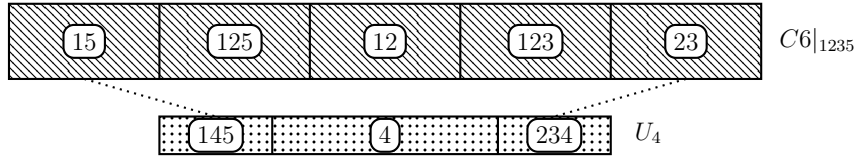

\begin{definition} \label{def:rigid}
Let $\code$ be a code on $n$ neurons.  
Let $\emptyset \neq \sigma \subseteq [n]$.  A pair 
$(\sigma, \bigcirc)$ 
is a {\em rigid structure} of $\code$ if one of the following holds:  
\begin{enumerate}
    \item $\bigcirc = \cap$, or
    \item $\bigcirc = \cup$, 
    and there exists a set 
    $\sigma'$ with 
    $\sigma \subseteq \sigma' \subseteq [n]$ such that the union $\cup_{i \in \sigma} U_i$ is convex in every closed-convex realization $\mathcal{U} = \{ U_i\}_{i \in \sigma'}$ of $\code|_{\sigma'}$.   
\end{enumerate}
\end{definition}

\begin{remark}
Definition~\ref{def:rigid}(1) lacks an extra condition, because the intersection $\cap_{i \in \sigma} U_i$ 
(unlike the union in Definition~\ref{def:rigid}(2))
is automatically convex in every convex realization.
\end{remark}

\begin{remark}\label{rem:sigma-vs-sigma'}
In Definition~\ref{def:rigid}(2), 
we can always choose $\sigma' = [n]$, but checking this condition is generally easier when $\sigma' \subsetneq [n]$.  Indeed, in this section, we are interested in proving that a code $\code$ has no closed-convex realization -- but the restricted code $\code|_{\sigma'}$ may have such realizations, and so may be easier to analyze.
\end{remark}

\begin{example} \label{ex:restrict-to-1-neuron}
Let $\code$ be a code on $n$ neurons.  If there is a neuron $j$ that is contained in every nonempty codeword of $\code$, then 
$([n], \cup)$
is a rigid structure.
Indeed, 
for such a code $\code$,
we have 
$\cup_{i \in [n]} U_i = U_j$ for every convex realization $\fullreal$. 
\end{example}

In general, it is challenging to check whether a given pair $(\sigma, \cup)$ is a rigid structure, but we succeed in obtaining a sufficient result in this direction (Proposition~\ref{prop:path-rigid}).  Before turning to that topic, our aim is to state and prove our criterion for non-closed-convexity (Theorem~\ref{thm:rigid}).  To do so, we need several definitions and lemmas.  We begin by formalizing how two rigid structures ``fit together''.

\begin{definition} \label{def:subcode-from-r-and-c}
Let $\code$ be a code on $n$ neurons. 
Assume that 
$(\rig, \bigcirc)$ 
and 
$(\conn, \bigtriangleup)$ 
are rigid structures of $\code$.  
The {\em \distinguished subcode} arising from 
$(\rig, \bigcirc)$ 
and 
$(\conn, \bigtriangleup)$, 
is the neural code consisting of 
all codewords $\sigma \in \code$ such that:
\begin{enumerate}
    \item if 
        $\bigcirc= \cup$, 
        then $\rig \cap \sigma \neq \emptyset$, 
    \item if 
        $\bigcirc= \cap$, 
        then $\rig \subseteq \sigma$, 
    \item if 
        $\bigtriangleup= \cup$, 
        then $\conn \cap \sigma \neq \emptyset$, and
    \item if 
        $\bigtriangleup= \cap$, 
        then $\conn \subseteq \sigma$.
\end{enumerate}
\end{definition}

\begin{remark} \label{rem:conditions-match}
In Definition~\ref{def:subcode-from-r-and-c}, 
the roles of $\rig$ and $\conn$ are symmetric, that is, conditions (1) and (2) match conditions (3) and (4), respectively.
\end{remark}

\begin{remark} \label{rem:no-empty-codeword-dist-code}
Distinguished subcodes do not contain the empty codeword. 
\end{remark}

Examples of distinguished subcodes appear later in this section (see Table~\ref{tab:5-codes}).


\begin{lemma}[Realization of distinguished subcode] \label{lem:subcode-from-r-and-c}
Let $\code$ be a code on $n$ neurons. 
Assume that 
$(\rig, \bigcirc )$ and $(\conn, \bigtriangleup )$ 
are rigid structures of $\code$.  
Let $\code'$ be the \distinguished
subcode arising from 
$(\rig, \bigcirc )$ and $(\conn, \bigtriangleup )$.
Let $\fullreal$ be a realization of $\code$ in some $\mathbb{R}^d$.  Let 
\[
\mathcal{U}'
    ~=~ \{ U_i'\}_{i \in [n]}
    ~:=~ \left\{~ U_i 
                \cap 
                \left( \bigcirc_{j \in \rig} U_j  \right)
                \cap 
                \left( \bigtriangleup_{k \in \conn} U_k  \right)
                ~
                \right\} _{i \in [n]}~.
\]
Then $\code'$ is the code realized by $\mathcal{U}'$, that is, $\code' = {\rm code}(\mathcal{U'}, 
X)$,
where $X= \cup_{i \in [n]} U_i'$.
\end{lemma} 

\begin{proof} 
Consider $\sigma \subseteq [n]$.  By definition, $\sigma \in \code'$ if and only if $\sigma \in \code$ and conditions~(1)--(4) in Definition~\ref{def:subcode-from-r-and-c} hold.  This is equivalent to the following conditions:
\[
\atomMacro{\calU}{\sigma} \neq \emptyset
    \quad \quad {\rm and} \quad \quad 
    \atomMacro{\calU}{\sigma} \subseteq \left( \bigcirc_{j \in \rig} U_j  \right)
    \quad \quad {\rm and} \quad \quad 
    \atomMacro{\calU}{\sigma} \subseteq \left( \bigtriangleup_{k \in \conn} U_k \right)~,
    \]
i.e.,  $\atomMacro{\calU}{\sigma} \cap 
\left( \bigcirc_{j \in \rig} U_j  \right)
\cap
\left( \bigtriangleup_{k \in \conn} U_k \right)
$ is nonempty.  Hence, $\sigma \in \code'$ if and only if $\sigma \in {\rm code}(\calU', X)$.
\end{proof}

\begin{definition} \label{def:order-forcing-graph}
The 
\textit{codeword-containment graph}
of a neural code $\code$ 
is the (undirected) graph with vertex set consisting of all codewords of $\code$ and edge set 
$\{ (\sigma,\tau)   \mid  \sigma \subsetneq \tau \textrm{ or } \sigma \supsetneq \tau \}$.
\end{definition}

\begin{example} \label{ex:order-forcing-gph}
The codeword-containment graph of $C6|_{1235}$ mirrors the top part of Figure~\ref{fig:C6-rigid}:

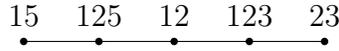
\begin{center}
    \begin{tikzpicture} [scale=1]
	\tikzstyle{vertex}=[circle,fill=black, inner sep = 1pt]
	\node[vertex, label=$15$] at (0,0) {};
	\node[vertex, label=$125$] at (1,0) {};
	\node[vertex, label=$12$] at (2,0) {};
	\node[vertex, label=$123$] at (3,0) {};
	\node[vertex, label=$23$] at (4,0) {};
    \draw (0,0) -- (4,0);    
\end{tikzpicture}
\end{center}
\end{example}

The following lemma is essentially due to Jeffs, Lienkaemper, and Youngs~\cite[Lemma 2.1]{order-forcing}. 
\begin{lemma} \label{lem:path-order-forcing-gph}
Let $\fullreal$ be a collection of closed, convex sets in $\mathbb{R}^d$.  Let $\code = {\rm code}(\calU, X)$, for some stimulus space $X \subseteq \mathbb{R}^d$ that contains every $U_i$.  Let $L$ be a line segment in~$\mathbb{R}^d$, and let 
$\atomMacro{\calU}{ \sigma_1 }, 
\atomMacro{\calU}{ \sigma_2 }, \dots, 
\atomMacro{\calU}{ \sigma_q }$
be the atoms that $L$ intersects (in order, from one endpoint of $L$ to the other, so an atom may appear more than once in this list).  Then 
$(\sigma_1,\sigma_2)$, 
$(\sigma_2,\sigma_3)$, $\dots$, 
$(\sigma_{q-1},\sigma_q)$ are edges in 
the codeword-containment graph of $\code$.
\end{lemma}

\begin{proof} 
Each intersection $U_i \cap L$ is either empty or a closed interval in $L$.  Label the set of all endpoints of these intervals as follows:

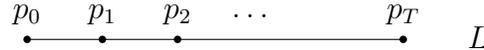
\begin{center}
    \begin{tikzpicture} [scale=1]
	\tikzstyle{vertex}=[circle,fill=black, inner sep = 1pt]
	\node[vertex, label=$p_0$] at (0,0) {};
	\node[vertex, label=$p_1$] at (1,0) {};
	\node[vertex, label=$p_2$] at (2,0) {};
	\node[ label=$\dots$] at (3,0) {};
	\node[vertex, label=$p_T$] at (5,0) {};
	\node[label=$L$] at (6,-0.4) {};
    \draw (0,0) -- (5,0);    
\end{tikzpicture}
\end{center}

By construction, each open interval $(p_j, p_{j+1} )$, for $j=0,1,\dots, T-1$, is contained in some atom $\atomMacro{\calU}{\nu_j}$ (with $\nu_j \in \code$).  So, it suffices to show that the atom containing the endpoint $p_j$, denoted by $\atomMacro{\calU}{\tau_j}$, satisfies $\nu_j \subseteq \tau_j$ (and so, by symmetry, $\nu_j \subseteq \tau_{j+1}$ also holds).

Let $i \in [n]$.  We consider three cases.  
If $p_j$ is the right endpoint of $U_i \cap L$, then $i \in (\tau_j \smallsetminus \nu_j)$.  
If $p_j$ is the left endpoint of $U_i \cap L$, then $i$ is in both $\tau_j $ and $ \nu_j$.  
If $p_j$ is not an endpoint of $U_i \cap L$, then $i$ is 
    either in both $\tau_j $ and $ \nu_j$ or in neither set.
We therefore obtain the desired containment $\nu_j \subseteq \tau_j$.
\end{proof}

The intuition behind the next result is shown in Figure~\ref{fig:intuition-behind-thm}: When the intersection between two rigid structures $\rig$ and $\conn$ is disconnected (more precisely, when the atoms they have in common generate a disconnected codeword-containment graph), then one of the structures cannot be convex and hence the code is not closed-convex.

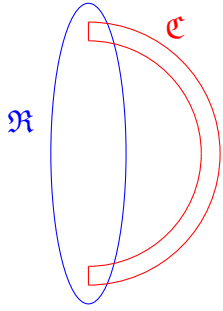
\begin{figure}[ht]
    \centering
    \begin{tikzpicture}
        \draw[color = blue] (2,2) ellipse (0.5 cm and 2 cm);
        \draw[color=red] (3.5,2) arc (0:90:1.5 cm);
        \draw[color=red] (3.5,2) arc (0:-90:1.5 cm);
        \draw[color=red] (3.75,2) arc (0:90:1.75 cm);
        \draw[color=red] (3.75,2) arc (0:-90:1.75 cm);
        \draw[color=red] (2,3.5) -- (2,3.75);
        \draw[color=red] (2,0.25) -- (2,0.5);
        \node[label = {\color{blue}$\mathfrak{R}$}] at (1.1,2) {};
        \node[label = {\color{red}$\mathfrak{C}$}] at (3.15,3.25) {};        
    \end{tikzpicture}
    \caption{Idea behind Theorem~\ref{thm:rigid}: 
    If the intersection of two rigid structures $\rig$ and $\conn$ is disconnected, then at least one of $\rig$ and $\conn$ is non-convex.}
    \label{fig:intuition-behind-thm}
\end{figure}

\begin{theorem}[Criterion for precluding closed-convexity]\label{thm:rigid}
Let $\code$ be a code on $n$ neurons. 
Assume that 
$(\rig, \bigcirc)$ 
and 
$(\conn, \bigtriangleup )$ 
are rigid structures of $\code$.  
Let $\code'$ be the \distinguished
subcode arising from 
$(\rig, \bigcirc )$ and $(\conn, \bigtriangleup )$.
If the codeword-containment graph of $\code'$ is disconnected, then $\code$ is not closed-convex.
\end{theorem}

\begin{proof}
Assume that 
$(\rig, \bigcirc )$ and $(\conn, \bigtriangleup )$ 
are rigid structures of $\code$, 
and assume that the resulting \distinguished
subcode, denoted by $\code'$, is disconnected.

Suppose for contradiction that there is a closed-convex realization $\fullreal$ of~$\code$. 
Let $\sigma$ and $\tau$ be 
codewords of~$\code$ in distinct connected components of the codeword-containment graph of~$\code'$.
Let $L$ be a line segment from a point in the atom
$\atomMacro{\realiz}{\sigma}$ to a point in the atom
$\atomMacro{\realiz}{\tau}$.
By Lemma~\ref{lem:subcode-from-r-and-c}, both endpoints of $L$ are in 
$\left( \bigcirc_{j \in \rig} U_j  \right)
                \cap 
                \left( \bigtriangleup_{k \in \conn} U_k  \right)$.
Also, by Lemma~\ref{lem:path-order-forcing-gph}
(and the fact that $\sigma$ and $\tau$ are not in the same connected component of the codeword-containment graph of $\code'$),
$L$ intersects some atom $\atomMacro{\mathcal{U}}{\nu}$ for some codeword $\nu \in \code$ that is not
a vertex of the codeword-containment graph of $\code'$.  Thus, $\nu$ is a non-codeword of $\code'$.  

Hence, by Lemma~\ref{lem:subcode-from-r-and-c}, 
the atom $\atomMacro{\mathcal{U}}{\nu}$ 
does not intersect 
$\bigcirc_{j \in \rig} U_j  $ or does not intersect $\bigtriangleup_{j \in \conn} U_j $.  
We conclude that $L$, which by assumption intersects 
 $\atomMacro{\mathcal{U}}{\nu}$, is not fully contained in 
$\bigcirc_{j \in \rig} U_j  $ or is not fully contained in $
                 \bigtriangleup_{j \in \conn} U_j $ (even though the endpoints of $L$ are in those sets).  Thus, $\bigcirc_{j \in \rig} U_j  $ or $
                 \bigtriangleup_{j \in \conn} U_j $ is not convex, and so 
    $(\rig, \bigcirc )$ or $(\conn, \bigtriangleup )$
    is not rigid (Definition~\ref{def:rigid}).  
                  This is a contradiction.  
\end{proof}

\begin{remark} \label{rem:connector}
The notation $\conn$
was chosen because we view one of the rigid structures as a ``connector'' of two ends of the other rigid structure $\rig$.
\end{remark}

The following corollary is the case of Theorem~\ref{thm:rigid} when the two rigid structures coincide.

\begin{corollary} \label{cor:rigid-structures-equal}
Let $\code$ be a code on $n$ neurons. 
Assume that  
$(\rig, \bigcirc)$ 
is a rigid structure of $\code$.  
Let $\code'$ be the neural code consisting of 
all codewords $\sigma \in \code$ such that:
(1) if 
        $\bigcirc= \cup$, 
        then $\rig \cap \sigma \neq \emptyset$; and (2) if 
        $\bigcirc= \cap$, 
        then $\rig \subseteq \sigma$.
If the codeword-containment graph of $\code'$ is disconnected, then $\code$ is not closed-convex.
\end{corollary}

The next result represents a special case of Theorem~\ref{thm:rigid}: when the codeword-containment graph of a code is a cycle and a certain triplewise-intersection condition holds.

\begin{theorem}[Cycle criterion for precluding closed-convexity] \label{thm:rigid-when-gph-is-cycle}
Let $\code$ be a code on $n$ neurons. 
Assume that the codeword-containment graph of $\code \smallsetminus \{ \emptyset \}$ is a cycle:

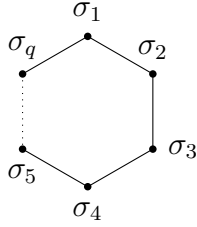
\begin{center}
    \begin{tikzpicture}[scale=1]
    \draw[black]
    ({sin(5*\y)},{cos(5*\y)})
    --
    (0,1)--({sin(\y)},{cos(\y)})--({sin(2*\y)},{cos(2*\y)})--({sin(3*\y)},{cos(3*\y)})--({sin(4*\y)},{cos(4*\y)}) ;
    \draw[dotted] 
    ({sin(5*\y)},{cos(5*\y)})--({sin(4*\y)},{cos(4*\y)});
	\tikzstyle{vertex}=[circle,fill=black, inner sep = 1pt]
	\node[vertex, label=$\sigma_1$] at (0,1) {};    	\node[vertex, label=$\sigma_2$] at ({sin(\y)},{cos(\y)}) {};   
	\node[vertex, label=right:$\sigma_3$] at ({sin( 2*\y)},{cos( 2*\y)}) {};   
	\node[vertex, label=below:$\sigma_4$] at ({sin( 3*\y)},{cos( 3*\y)}) {};   
	\node[vertex, label=below:$\sigma_5$] at ({sin( 4*\y)},{cos( 4*\y)}) {};   
	\node[vertex, label=$\sigma_q$] at ({sin( 5*\y)},{cos( 5*\y)}) {};   
	\end{tikzpicture}
\end{center}

If, additionally, 
   for 
    all $i=1,2,\dots,q$, 
    the intersection $\sigma_i \cap \sigma_{i+1} \cap \sigma_{i+2}$ is nonempty (where 
    $\sigma_{q+1}:=\sigma_1$ and $\sigma_{q+2}:= \sigma_2$), 
    then $\code$ is not closed-convex.
\end{theorem}

The proof of Theorem~\ref{thm:rigid-when-gph-is-cycle} appears in Section~\ref{sec:proof-cycle}; we first need results (in the next subsection) that help us check whether a given a subset of neurons is rigid.  
Indeed, in general, being able to check rigidity is the main difficulty in applying Theorem~\ref{thm:rigid}.

\begin{example} \label{ex:cycle}
The codeword-containment graphs of $C15 \smallsetminus \{\emptyset \}$ and $\CruzCode \smallsetminus \{\emptyset \}$ are, respectively, the following cycle graphs:
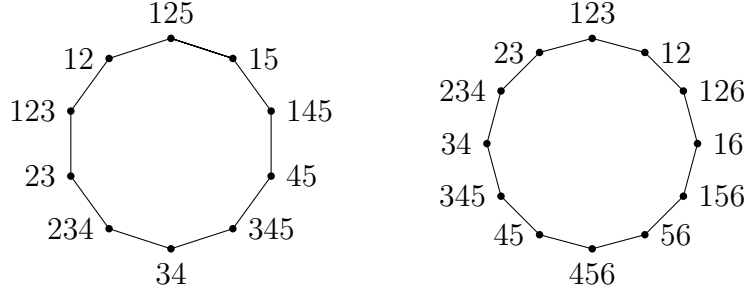
\begin{center}
    \begin{tikzpicture}[scale=1.4]
    \draw[black]
    (0,1)--({sin(\z)},{cos(\z)})--({sin(2*\z)},{cos(2*\z)})--({sin(3*\z)},{cos(3*\z)})--({sin(4*\z)},{cos(4*\z)})
    --
    ({sin(5*\z)},{cos(5*\z)})
    --
    ({sin(6*\z)},{cos(6*\z)})
    --
    ({sin(7*\z)},{cos(7*\z)})
    --
    ({sin(8*\z)},{cos(8*\z)})
    --
    ({sin(9*\z)},{cos(9*\z)})
    --
    ({sin(10*\z)},{cos(10*\z)})
    --
    ({sin(11*\z)},{cos(11*\z)})
    --
	(0,1);
     ;
	\tikzstyle{vertex}=[circle,fill=black, inner sep = 1pt]
	\node[vertex, label=$125 $] at (0,1) {};    	
	\node[vertex, label=right:$15 $] at ({sin(\z)},{cos(\z)}) {};   
	\node[vertex, label=right:$145 $] at ({sin( 2*\z)},{cos( 2*\z)}) {};   
	\node[vertex, label=right:$ 45$] at ({sin( 3*\z)},{cos( 3*\z)}) {};   
	\node[vertex, label=right:$ 345$] at ({sin( 4*\z)},{cos( 4*\z)}) {};   
	\node[vertex, label=below:$ 34$] at ({sin( 5*\z)},{cos( 5*\z)}) {};   
	\node[vertex, label=left:$ 234$] at ({sin( 6*\z)},{cos( 6*\z)}) {};   
	\node[vertex, label=left:$ 23$] at ({sin( 7*\z)},{cos( 7*\z)}) {};   
	\node[vertex, label=left:$ 123$] at ({sin( 8*\z)},{cos( 8*\z)}) {};   
	\node[vertex, label=left:$ 12$] at ({sin( 9*\z)},{cos( 9*\z)}) {};   
      \draw[black]
    (4,1)--({4+sin(\w)},{cos(\w)}) -- ({4+sin(2*\w)},{cos(2*\w)})--({4+sin(3*\w)},{cos(3*\w)})--({4+sin(4*\w)},{cos(4*\w)})
    --
    ({4+sin(5*\w)},{cos(5*\w)})
    --
    ({4+sin(6*\w)},{cos(6*\w)})
    --
    ({4+sin(7*\w)},{cos(7*\w)})
    --
    ({4+sin(8*\w)},{cos(8*\w)})
    --
    ({4+sin(9*\w)},{cos(9*\w)})
    --
    ({4+sin(10*\w)},{cos(10*\w)})
    --
    ({4+sin(11*\w)},{cos(11*\w)})
    --
	(4,1);
     ;
	\tikzstyle{vertex}=[circle,fill=black, inner sep = 1pt]
	\node[vertex, label=$123 $] at (4,1) {};    	
	\node[vertex, label=right:$12 $] at ({4+sin(\w)},{cos(\w)}) {};   
	\node[vertex, label=right:$126 $] at ({4+sin( 2*\w)},{cos( 2*\w)}) {};   
	\node[vertex, label=right:$ 16$] at ({4+sin( 3*\w)},{cos( 3*\w)}) {};   
	\node[vertex, label=right:$ 156$] at ({4+sin( 4*\w)},{cos( 4*\w)}) {};   
	\node[vertex, label=right:$ 56$] at ({4+sin( 5*\w)},{cos( 5*\w)}) {};   
	\node[vertex, label=below:$ 456$] at ({4+sin( 6*\w)},{cos( 6*\w)}) {};   
	\node[vertex, label=left:$ 45$] at ({4+sin( 7*\w)},{cos( 7*\w)}) {};   
	\node[vertex, label=left:$ 345$] at ({4+sin( 8*\w)},{cos( 8*\w)}) {};   
	\node[vertex, label=left:$ 34$] at ({4+sin( 9*\w)},{cos( 9*\w)}) {};   
	\node[vertex, label=left:$ 234$] at ({4+sin( 10*\w)},{cos( 10*\w)}) {};   
	\node[vertex, label=left:$ 23$] at ({4+sin( 11*\w)},{cos( 11*\w)}) {};     
	\end{tikzpicture}
\end{center}
All triplewise-intersection conditions are easy to check (e.g., $125 \cap 15 \cap 145 = 15 \neq \emptyset $). Hence,  Theorem~\ref{thm:rigid-when-gph-is-cycle} implies that both $C15$ and $\CruzCode$ are non-closed-convex.
\end{example}

\begin{remark} \label{rmk:pinwheel}
When a code's codeword-containment graph is a cycle,
as in Theorem~\ref{thm:rigid-when-gph-is-cycle},
this is often revealed in a ``pinwheel'' in an open-convex realization of the code (if the code is open-convex).  Such a pinwheel is depicted in 
Figure~\ref{fig:pinwheel-infinte-family} in
the proof of Theorem~\ref{thm:infinite-family-non-closed},  and also
 in~\cite[Appendix~B]{goldrup2020classification} (for the code $C15$) and~\cite[Figure~2(a)]{cruz2019open} (for the code $\CruzCode$).
\end{remark}


\subsection{A sufficient condition for rigidity} \label{sec:path}
In this subsection, we show that rigid structures arise whenever the codeword-containment graph of a code -- or a subcode arising in a specific way from a subset of neurons -- is a path (Lemma~\ref{lem:case-of-path-prop-with-all-neurons} and Proposition~\ref{prop:path-rigid}). 
The proof of 
Proposition~\ref{prop:path-rigid}, which appears at the end of this subsection,
relies on 
a minimum-distance argument (see Lemma~\ref{lem:hinge}) and is similar to (and inspired by) related proofs in~\cite{cruz2019open, goldrup2020classification}.

\begin{proposition}[Criterion for rigidity] \label{prop:path-rigid}
    Let $\code$ be a code on $n$ neurons.
    Let $\emptyset \neq \rig \subseteq [n]$. 
    Let $\Gamma$ be the codeword-containment graph of 
    the code 
    $\{ \tau \in \code \mid \tau \cap \rig \neq \emptyset \}$. 
    Assume $\Gamma$ is a path:
    
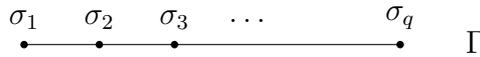
\begin{center}
    \begin{tikzpicture} [scale=1]
	\tikzstyle{vertex}=[circle,fill=black, inner sep = 1pt]
	\node[vertex, label=$\sigma_1$] at (0,0) {};
	\node[vertex, label=$\sigma_2$] at (1,0) {};
	\node[vertex, label=$\sigma_3$] at (2,0) {};
	\node[ label=$\dots$] at (3,0) {};
	\node[vertex, label=$\sigma_q$] at (5,0) {};
	\node[label=$\Gamma $] at (6,-0.4) {};
    \draw (0,0) -- (5,0);    
\end{tikzpicture}
\end{center}
If, 
   for 
    all $i=1,2,\dots,q-2$, 
    the intersection $\sigma_i \cap \sigma_{i+1} \cap \sigma_{i+2}$ is nonempty, 
then $(\rig, \cup )$ 
    is a rigid structure of~$\code$.  
\end{proposition}

\begin{example}\label{ex:goldphil-C6-rig}
We claim that for the code
$C6 = \{ {\bf 125,234,145,123},4,23,15, 12,\emptyset\}$
from Example~\ref{ex:goldrup-phillipson}, 
$(1235, \cup) $ is a rigid structure.
To see this, we apply Proposition~\ref{prop:path-rigid} to $C6|_{\{1,2,3,5\}}$
with $\rig=1235$: the graph $\Gamma$ is the path $15 - 125 - 12 - 123 - 23$, and the intersections 
 $15 \cap 125 \cap 12 = 1$, 
  $125 \cap 12 \cap 123 =12$, and 
   $12 \cap 123 \cap 23=2$ are all nonempty. 
\end{example}

In 
Section~\ref{sec:examples-rigid}, 
we analyze more codes using Proposition~\ref{prop:path-rigid}.  

\begin{remark}[Triplewise-intersection condition] \label{rem:triplewise-intersection-condition}
The triplewise-intersection condition in 
Theorem~\ref{thm:rigid-when-gph-is-cycle} and
Proposition~\ref{prop:path-rigid} --
$( \sigma_i \cap \sigma_{i+1} \cap \sigma_{i+2}) \neq \emptyset$ --
can be viewed as a ``brace'' condition that forbids a closed-convex realization from ``bending'' at the atom of $\sigma_{i+1}$.  This condition 
can not be removed.  Indeed, consider the code $\{ {\bf 12}, 1, 2, \emptyset \}$. 
For $\rig=\{1,2\}$, the graph $\Gamma$ is a path, the intersection $1 \cap 12 \cap 2$ is empty, and the following closed-convex realization shows that $(\rig, \cup)$ is not a rigid structure:
    
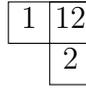
\begin{center}
    \begin{tikzpicture} [scale=0.55]
	\node[label=$1$] at (0.5, 0.85) {};
	\node[label=$12$] at (1.5, 0.85) {};
	\node[label=$2$] at (1.5, -0.15) {};
    \draw (0,1) rectangle (2,2);
    \draw (1,0) rectangle (2,2);    
\end{tikzpicture}
\end{center}
\end{remark}

The rest of this subsection is dedicated to proving Proposition~\ref{prop:path-rigid}, beginning with the following lemma. 

\begin{lemma} \label{lem:hinge}
Let $V_1, V_2, W_1, W_2$ be closed sets in some $\mathbb{R}^d$.  If $V_1$ and $V_2$ are convex, and: \begin{itemize}
    \item[(i)]  $ V_1 \cap V_2 \neq \emptyset$,
    \item[(ii)] $ V_1 \subseteq (V_2 \cup W_1)$
        and $V_2 \subseteq (V_1 \cup W_2)$, and
    \item[(iii)] $V_1 \cap V_2 \cap W_1 \cap W_2  = \emptyset$~,
\end{itemize}
then $V_1 \cup V_2$ is convex.
\end{lemma}

\begin{proof}
Let $p\in V_1$ and $q\in V_2$, and let $L = \overline{pq}$ be the line segment between $p$ and $q$. Assume for a contradiction that $L\not\subseteq V_1 \cup V_2$ (and in particular $p \notin V_2$ and $q \notin V_1$).

First, we claim that $L \cap V_1 \cap V_2 = \emptyset$. 
To see this, assume that there exists $r \in L \cap V_1 \cap V_2$. Then, by convexity of $V_1$ and $V_2$, we have $\overline{pr}\subseteq V_1$ and $\overline{qr}\subseteq V_2$.
This implies that $L = \overline{pr} \cup \overline{qr}$ is contained in $V_1 \cup V_2$, which contradicts our hypothesis.  So, our claim is true.

The line segment $L$ is compact. Also, $V_1 \cap V_2$ is closed and, by $(i)$, nonempty. Hence, there exists $x\in( V_1 \cap V_2)$ 
that realizes the (positive) distance between $L$ and $V_1 \cap V_2$.  That is, 
${\rm dist} (x,L) \leq {\rm dist} (y,L)$ for all $ y\in (V_1 \cap V_2)$.  

Define $L_1 = \overline{px}$ and $L_2= \overline{qx}$. 
We claim that the three points $p,q,x$ are not collinear. Indeed, none of the points is on the line segment defined by the other two: (1) $x\notin L$ because ${\rm dist} (x,L)>0$, 
(2) $p\not\in L_2$ because $p \notin V_2 \supseteq L_2$, and, symmetrically, (3) $q\not\in L_1$.  Hence, $p,q,x$ define 
a triangle, which we depict here together with $V_1$ and $V_2$: 

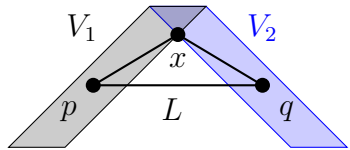
\begin{center}
    \begin{tikzpicture}[scale=1.5]
	\tikzstyle{vertex}=[circle, fill=black, inner sep = 2pt];
	\draw[blue,  fill=blue, fill opacity=0.2] 
	(0,0) -- (0.5,0) -- (1.75,-1.25) -- (1.25,-1.25) -- (0,0)  ;
    \draw[black,  fill=black, fill opacity=0.2] 
	(0,0) 
	-- (0.5,0) 
	-- (-0.75,-1.25) -- (-1.25,-1.25) -- (0,0) ;
	\node[label=${\color{blue} V_2 }$] at (1, -0.5) {};
	\node[label=$V_1$] at (-0.6, -0.5) {};
    \node[vertex, label=below:$x$] at (0.25, -0.25) {};
    \node[vertex, label=below left:$p$] at (-0.5, -0.7) {};
    \node[vertex, label=below right:$q$] at (1, -0.7) {};
    \draw [thick] (0.25, -0.25) -- (-0.5, -0.7) -- (1, -0.7) -- (0.25, -0.25);
	\node[label=$L$] at (0.2,-1.2) {};
    \end{tikzpicture}
\end{center}

Next, we claim that $L_1 \cap V_2 = \{x\}$.  Indeed, any point $x \neq y \in L_1 \subseteq V_1$ satisfies ${\rm dist} (y,L) < {\rm dist} (x,L)$ , and so (by construction of $x$) such a point $y$ is not in $V_2$.  

The line segment $L_1$ is covered by $V_2$ and $W_1$, because $L_1 \subseteq V_1$ and by $(ii)$.  However, by the above claim, $V_2$ covers only an endpoint of $L_1$; so, (because $W_1$ is closed) $W_1$ must cover all of $L_1$.  We conclude that $x\in W_1$ and so (by symmetry) $x \in W_2$ also.  Thus,
$x \in V_1 \cap V_2 \cap W_1 \cap W_2 $, 
which contradicts $(iii)$.
\end{proof}

\begin{remark}\label{rem:lemma-false-if-open-sets}
Lemma~\ref{lem:hinge} becomes false if ``closed sets'' is replaced by ``open sets''.
\end{remark}

The following lemma pertains to when the codeword-containment graph is a path.
\begin{lemma} 
    \label{lem:path-interval-condition-and-alternating-containments}
    Let $\code$ be a code on $n$ neurons.
    Let $\Gamma$ be the codeword-containment graph of the code
    $\code \smallsetminus \{ \emptyset\}$. 
    Assume that $\Gamma$ is a path:
    
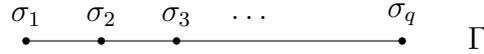
\begin{center}
    \begin{tikzpicture} [scale=1]
	\tikzstyle{vertex}=[circle,fill=black, inner sep = 1pt]
	\node[vertex, label=$\sigma_1$] at (0,0) {};
	\node[vertex, label=$\sigma_2$] at (1,0) {};
	\node[vertex, label=$\sigma_3$] at (2,0) {};
	\node[ label=$\dots$] at (3,0) {};
	\node[vertex, label=$\sigma_q$] at (5,0) {};
	\node[label=$\Gamma $] at (6,-0.4) {};
    \draw (0,0) -- (5,0);    
\end{tikzpicture}
\end{center}
Then the following conditions hold:
\begin{enumerate}
    \item 
 {\bf Interval Condition.}  If $\code$ is closed-convex, then 
 for all $i \in [n]$, either:
 \begin{itemize}
     \item no codeword of $\code$ contains $i$, or  
     \item there exists a ``left endpoint'' $L_i$ and a ``right endpoint'' $R_i$ with $1 \leq L_i \leq R_i \leq q$ such that 
     the codewords between the two endpoints are precisely those containing $i$ (that is,
     $ L_i \leq \ell \leq R_i \Longleftrightarrow i \in \sigma_{\ell}$).
 \end{itemize}
    \item {\bf Alternating-Containment Condition.} The containment relations between subsequent codewords alternate, so that one of the following holds:
\begin{align} \label{eq:containments-alternate}
 \sigma_1 \subsetneq \boldsymbol{\sigma_2 } \supsetneq \sigma_3 \subsetneq  \boldsymbol{\sigma_4 } \supsetneq  \dots  
            \quad \textrm{or} \quad
  \boldsymbol{\sigma_1 } \supsetneq \sigma_2 \subsetneq  \boldsymbol{\sigma_3 } \supsetneq {\sigma_4} \subsetneq \dots  
\end{align}
Here, the maximal codewords of $\code$ are indicated in bold.  
\end{enumerate}

\end{lemma}

\begin{proof}
Assume that the Interval Condition does not hold.  
Then, 
there exists a closed-convex realization $\fullreal$ of $\code$, 
and, additionally,
for some neuron $i \in [n]$,
there exist three codewords $\sigma_{\ell_1}$,
$\sigma_{\ell_2}$, 
$\sigma_{\ell_3}$ with $\ell_1 < \ell_2 < \ell_3$ such that $i \in (\sigma_{\ell_1} \cap \sigma_{\ell_3})$ but $i \notin \sigma_{\ell_2}$.  Then, by Lemma~\ref{lem:path-order-forcing-gph} and the construction of $\Gamma$, 
every line segment from the atom of $\sigma_{\ell_1}$ to the atom of $\sigma_{\ell_3}$ is not completely contained in $U_i$. Hence, $U_i$ is non-convex, which is a contradiction.
%

Next, we verify the Alternating-Containment Condition. 
If 
the containments 
$\sigma_{r-1} \subseteq \sigma_r \subseteq \sigma_{r+1}$
or 
$\sigma_{r-1} \supseteq \sigma_r \supseteq \sigma_{r+1}$
hold, for some $r$, then $(\sigma_{r-1}, \sigma_{r+1})$ is an edge of $\Gamma$, which 
contradicts the fact that $\Gamma$ is a path.
\end{proof}

We use Lemmas~\ref{lem:hinge} and~\ref{lem:path-interval-condition-and-alternating-containments} to prove the 
next result, 

which is the special case of Proposition~\ref{prop:path-rigid} when the rigid structure $\rig$ is the set of all $n$ neurons.  In turn, Lemma~\ref{lem:case-of-path-prop-with-all-neurons} will be used to prove Proposition~\ref{prop:path-rigid}.

\begin{lemma} 
    \label{lem:case-of-path-prop-with-all-neurons}
    Let $\code$ be a code on $n$ neurons.
    Let $\Gamma$ be the codeword-containment graph of the code
    $\code \smallsetminus \{ \emptyset\}$. 
    Assume that $\Gamma$ is a path:
    
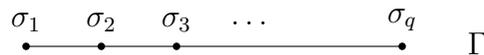
\begin{center}
    \begin{tikzpicture} [scale=1]
	\tikzstyle{vertex}=[circle,fill=black, inner sep = 1pt]
	\node[vertex, label=$\sigma_1$] at (0,0) {};
	\node[vertex, label=$\sigma_2$] at (1,0) {};
	\node[vertex, label=$\sigma_3$] at (2,0) {};
	\node[ label=$\dots$] at (3,0) {};
	\node[vertex, label=$\sigma_q$] at (5,0) {};
	\node[label=$\Gamma $] at (6,-0.4) {};
    \draw (0,0) -- (5,0);    
\end{tikzpicture}
\end{center}
If, 
   for 
    all $i=1,2,\dots,q-2$, 
    the intersection $\sigma_i \cap \sigma_{i+1} \cap \sigma_{i+2}$ is nonempty, 
then $([n],\cup)$ 
    is a rigid structure of~$\code$.  
\end{lemma}

\begin{proof} 
Let $\fullreal$ be a closed-convex realization of $\code$.  We must show $\cup_{i \in [n] } U_i$ is convex. 

By Lemma~\ref{lem:path-interval-condition-and-alternating-containments}, 
for all $i \in [n]$, either $U_i$ is empty or there exist left and right endpoints, $L_i$ and $R_i$, respectively, with $1 \leq L_i \leq R_i \leq q$, such that 
     $ L_i \leq \ell \leq R_i \Longleftrightarrow i \in \sigma_{\ell}$.
Lemma~\ref{lem:path-interval-condition-and-alternating-containments}
also implies that subsequent codewords alternate, as in~\eqref{eq:containments-alternate}.
 

When we have containments $ \sigma_{r-1} \subsetneq \boldsymbol{\sigma_r } \supsetneq \sigma_{r+1} $,
the index $r$ of the maximal codeword $\sigma_r$ is a right endpoint of all $U_i$ with $i \in (\sigma_r \smallsetminus \sigma_{r+1} )$ and is a 
left
endpoint of all $U_j$ with $j \in (\sigma_r \smallsetminus \sigma_{r-1} )$.  Furthermore, 
indices $\ell$ of non-maximal codewords $\sigma_{\ell}$ are not endpoints, unless $\ell = 1$ or $\ell= q$.  We will use these facts below.
  
If $q=1$, then $U_i=U_j$ for all $i,j \in \sigma_1 $
and, additionally, $U_k = \emptyset$ for all $k \in [n] \smallsetminus \sigma_1$; hence the union $\cup_{i \in [n]} U_i$ is convex.  
 If $q=2$, then we have $\sigma_1 \subsetneq \boldsymbol{\sigma_2  }$ (or the reverse containment, which is symmetric), and so $\cup_{i \in [n]} U_i = U_j$ for any $j \in \sigma_1$; hence, this union is convex.
So, assume for the rest of the proof that $q \geq 3.$

We relabel the $U_i$'s as follows:
\begin{itemize}
    \item Pick $U_1$ to be a ``leftmost'' receptive field (i.e., $L_1=1$) of maximal width (i.e., $R_1$ is maximal among all $U_i$ with $L_i=1$).  Notice that $R_1>1$ holds: the triplewise-intersection condition (and the $q\geq 3 $ assumption) imply that there is some $i \in (\sigma_1  \cap \sigma_2 \cap \sigma_3)$, and so $L_i=1$ and $R_i \geq 3$ for this $i$.  
    \begin{itemize}
        \item     If $R_1 = q$, then relabel all remaining $U_i$'s arbitrarily.  
        \item   Otherwise, continue to the next step.
    \end{itemize}
    \item Pick $U_2$ 
    to be some $U_i$ such that $i \in (\sigma_{R_1-1} \cap \sigma_{R_1} \cap \sigma_{R_1+1})$ 
    (the triplewise-intersection condition implies that such an $i$ exists, and 
    the inequalities $1 < R_1 < q$ guarantee that $1 \leq R_1-1$ and $R_1+1 \leq q$).  Thus, $L_2 < R_1 < R_2$.
    \begin{itemize}
        \item   If $R_2 = q$, then relabel all remaining $U_i$'s arbitrarily.  
        \item   Otherwise, continue to the next step.
    \end{itemize}
    \item Successively pick $U_3, U_4, \dots, U_m$ (for some $m\leq n$) in the same way (that is, $U_k$, for $k \geq 3$, is chosen to be some $U_i$ such that $i \in (\sigma_{R_{i-1}-1} \cap \sigma_{R_{i-1}} \cap \sigma_{R_{i-1}+1})$), so that $R_2 < R_3 < \dots < R_m = q$ and $L_3< R_2$, 
    $L_4< R_3, \dots$, 
    $L_m< R_{m-1}$.  Relabel the remaining $U_i$'s arbitrarily.
\end{itemize}

By construction, $\cup_{i=1}^m U_i = \cup_{i=1}^n U_i$.
So, it suffices to prove (by induction) that the union $\cup_{i=1}^k U_i$ is convex for all $k=1,2,\dots, m$.

The base case ($k=1$) is true by assumption.  
For induction, assume that $\cup_{i=1}^k U_i$ is convex for some $1 \leq k \leq m-1$.  We will show that $\cup_{i=1}^{k+1} U_i$ is convex by applying Lemma~\ref{lem:hinge}.  To this end, let 
$V_1= \cup_{i=1}^k U_i$ and 
$V_2=U_{k+1}$.  Both $V_1$ and $V_2$ are closed and convex by hypothesis, and we depict them  schematically here:

\begin{center}
    \begin{tikzpicture}[scale=1]
    \fill[black,  fill=black, fill opacity=0.2] (0,0) rectangle (7,1);
    \fill[blue,  fill=blue, fill opacity=0.2] (3,0) rectangle (10,1);
	\node[label={${\color{blue} V_2 =  U_{k+1} }$}] at (7, 1) {};
	\node[label={$V_1 = \cup_{i=1}^k U_i $}] at (3,1) {};
	\node[label={$L_1=1$}] at (0.8, 0) {};
	\node[label={$L_{k+1}$}] at (3.7, 0) {};
	\node[label={$R_{k}$}] at (6.3, 0) {};
	\node[label={$R_{k+1}$}] at (9.5, 0) {};
    \end{tikzpicture}
\end{center}

By construction, $L_1=1< L_{k+1} < R_k< R_{k+1}$, so $V_1 \nsubseteq V_2$ and $V_1 \nsupseteq V_2$.  The intersection $V_1 \cap V_2$ is nonempty: it consists of all atoms between the endpoints $L_{k+1}$ and $R_k$.  

Next, we construct a closed set $W_1$ that covers all atoms to the left of $L_{k+1}$.  Let 
\[
    \mathcal{L} ~=~ \{ \ell \mid 1 \leq \ell \leq L_{k+1} ~\mathrm{and}~ \sigma_{\ell}  \textrm{ is a maximal codeword of } \code\}~.
\]
Let $\ell \in \mathcal{L}$. If $\ell=1$ (so, $\sigma_1$ is a maximal codeword), pick $i_{\ell} \in (\sigma_1 \smallsetminus \sigma_2)$; hence, $U_{i_\ell}$ contains the atom of $\sigma_1$. 
If $\ell \neq 1$, proceed as follows.  Pick $i_{\ell} \in  (\sigma_{\ell-1} \smallsetminus \sigma_{\ell + 1} ) $ (which exists because otherwise $\sigma_{\ell -1 } \subseteq \sigma_{\ell+1}$, 
which would contradict the fact that $\Gamma$ is a path).
Hence, $i_{\ell} \in  \sigma_{\ell-1}  \subseteq \sigma_{\ell}$ and so $U_{i_{\ell}}$ contains the atoms of $\sigma_{\ell-1}$ and $\sigma_{\ell}$.  
On the other hand, $i_{\ell} \notin \sigma_{\ell+1 }$ implies that the right endpoint of 
$U_{i_{\ell}}$  is $R_{i_{\ell}} = \ell \leq L_{k+1}$.

It follows that $W_1:= \cup_{\ell \in \mathcal{L}} U_{i_{\ell}}$ is a closed set such that $V_1 \subseteq ( V_2 \cup W_1)$.

We define $W_2$ similarly.  Let 
\[
    \mathcal{R} ~=~ \{ r \mid R_k \leq r \leq R_{k+1} ~\mathrm{and}~ \sigma_{r}  \textrm{ is a maximal codeword of } \code\}~.
\]
Let $r \in \mathcal{R}$. 
If $r=q$, pick $j_{r} \in (\sigma_q \smallsetminus \sigma_{q-1})$; hence, $U_{j_r}$ contains the atom of $\sigma_r$. 
If $r< q$, proceed as follows.  
Pick $j_{r} \in  (\sigma_{r+1} \smallsetminus \sigma_{r - 1} ) $ (which exists because otherwise $\sigma_{r+1 } \subseteq \sigma_{r-1}$, which is a contradiction).
Hence, $j_{r} \in  \sigma_{r+1}  \subseteq \sigma_{r}$ and so $U_{j_{r}}$ contains the atoms of $\sigma_{r}$ and $\sigma_{r+1}$.  
On the other hand, $j_{r} \notin \sigma_{r-1 }$ implies that $L_{j_{r}} = r \geq R_{k+1}$. 
Hence, $W_2:= \cup_{r \in \mathcal{R}} U_{j_{r}}$ is a closed set such that $V_2 \subseteq ( V_1 \cup W_2)$.

Now $V_1,V_2,W_1,W_2$ satisfy hypotheses $(i)$ and $(ii)$ of 
Lemma~\ref{lem:hinge}.  It remains only to check~$(iii)$. 
We saw above that the right endpoint of every $U_{i_{\ell}}$ is at most $L_{k+1}$, and the left endpoint of every $U_{j_r}$ is at least $R_k$.  Also, $L_{k+1} < R_k$. Thus, 
$(W_1 \cap W_2)
=
(\cup_{\ell \in \mathcal{L}} U_{i_{\ell} } ) 
\cap 
(\cup_{r \in \mathcal{L}} U_{j_{r}} )
= \emptyset
$,  and so $(iii)$ holds. 
Thus, 
the union $\cup_{i=1}^{k+1} U_i$ is convex.
\end{proof}

One might hope to now prove Proposition~\ref{prop:path-rigid} by applying Lemma~\ref{lem:case-of-path-prop-with-all-neurons} to a restricted code $\code|_{\rig}$.  However, this does not always work.  For instance, for the code $\CruzCode$ with $\rig=1456$, the restricted code $\CruzCode |_{\rig}$ generates a codeword-containment graph $\Gamma$ that is not a path: 
the codewords $12, 126, 16, 156$, which form a path in $\CruzCode$, become $ 1, 16, 16, 156$ in $\CruzCode|_{\rig}$, and so a vertex is lost and a non-path edge is added in $\Gamma$.  To circumvent this problem, we introduce redundant neurons $i_{12}$ and $ i_{126}$ that ``remember'' the original atoms so that, with the redundant neurons, the restricted codewords -- 
$1 i_{12},  16 i_{126} i_{12}, 16 , 156$ -- again form a path.

The following definition explains how to add redundant neurons 
(this is in contrast to the focus, in~\cite{jeffs2020morphisms}, on removing redundant neurons).

\begin{definition} \label{def:redundant}
Let $\code$ be a code on $n$ neurons, and let $\sigma \subseteq [n]$.  
The {\em code obtained by adding a redundant neuron $i_0$ for $\sigma$} is the code on $n+1$ neurons obtained from $\code$ by replacing each codeword $\tau \in \code$ that contains $\sigma$ with $\tau \cup\{i_0\}$.
\end{definition}

\begin{lemma}[Adding redundant neurons preserves convexity and codeword-containment] \label{lem:redundant}
Let $\code$ be a code on $n$ neurons, and let $\sigma \subseteq [n]$.  Let $\widetilde{\code}$ be obtained from $\code$ by adding a redundant neuron $i_{0}$ for $\sigma$. 
Then:
\begin{enumerate}
    \item 
    If 
    $\{U_i\}_{i \in [n]}$
    is a realization (or closed-convex realization) of $\code$, 
    then 
    $\{U_i\}_{i \in [n]} \cup \{U_{i_0}:=U_{\sigma} \}$
    is a realization (or closed-convex realization) of $ \widetilde{\code}$.
    Conversely, 
    if 
    $\{U_i\}_{i \in [n]} \cup \{U_{i_0}\}$
    is a realization (or closed-convex realization) of $ \widetilde{\code}$, 
    then 
    $\{U_i\}_{i \in [n]}$
    is a realization (or closed-convex realization) of $\code$.
    \item The following is an inclusion-preserving bijection:
    \begin{align*}
            \widetilde{\code} & \to \code \\
            \widetilde{\tau}& \mapsto (\widetilde{\tau} \smallsetminus \{i_0\})
    \end{align*}
\end{enumerate}
\end{lemma}
\begin{proof}
Part~(1) is due to Jeffs~\cite{jeffs2020morphisms}, and part~(2) is straightforward to check. \end{proof}

\begin{lemma}[Redundant neurons and rigid structures] \label{lem:redundant-rigid}
Let $\code$ be a code on $n$ neurons.  
Consider subsets $\rig \subseteq [n]$ and $\sigma \subseteq [n]$ such that
$\sigma \cap \rig \neq \emptyset$.  
Let $\widetilde{\code}$ be obtained from $\code$ by adding a redundant neuron $i_{0}$ for $\sigma$.
Then $\rig$ is a rigid structure of $\code$ if and only if $\rig \cup \{ i_0\}$ is a rigid structure of $\widetilde{\code}$.
\end{lemma}
\begin{proof}
The assumption $\sigma \cap \rig \neq \emptyset$ implies that, 
in every realization $\{U_i\}_{i \in [n] } \cup \{ U_{i_0} \}$ of $\widetilde{\code}$, we have $\cup_{j \in \rig} U_j = U_{i_0} \cup (\cup_{j \in \rig} U_j)$. 
Now the desired result follows from Lemma~\ref{lem:redundant}(1). \end{proof}

We are now ready to prove our criterion for rigidity (Proposition~\ref{prop:path-rigid}).

\begin{proof}[Proof of Proposition~\ref{prop:path-rigid}]
%
Let $\widetilde{\code}$ be obtained from $\code$ by adding a redundant neuron $i_{\omega}$ for every codeword $\omega \in \code$ that contains 
at least one neuron of $\rig$ (that is, $\omega \cap \rig \neq \emptyset$)
and 
at least one neuron outside of $\rig$ (that is, $\omega \nsubseteq \rig$). 
Let 
$\mathcal{I}$ be the set of all such redundant neurons $i_{\omega}$ added to $\code$.

Consider the following functions between subsets of the codes $\code$, $\widetilde{\code}$, and $\widetilde{\code}|_{ \rig \cup \mathcal{I}}$:
\begin{align} \label{eq:bijections}
    \{ \tau \in \code \mid \tau \cap \rig \neq \emptyset \}
    \xleftarrow{\phi}
    \{ \widetilde{ \tau} \in \widetilde{\code} \mid \widetilde{\tau} \cap \rig \neq \emptyset \}
        \xrightarrow{\psi}
        &
    \{ \widetilde{ \tau} \cap (\rig \cup \mathcal{I})  \mid 
    \widetilde{ \tau} \in \widetilde{\code} \textrm{ and }
    \widetilde{\tau} \cap \rig \neq \emptyset \}    
    \\
     & \quad 
     ~=~ \left( \widetilde{\code}|_{ \rig \cup \mathcal{I} } \right) \smallsetminus \{ \emptyset \}~,
     \label{eq:equality-codes}
\end{align}
given by $\phi( \widetilde{\tau}) := (\widetilde{\tau} \cap [n])$, and 
$\psi(\widetilde{ \tau}) := \widetilde{ \tau} \cap (\rig \cup \mathcal{I})$.  (We will prove the equality~\eqref{eq:equality-codes}  below.)

The map $\phi$ is an inclusion-preserving bijection by (repeated application of) Lemma~\ref{lem:redundant}(2).
The map $\psi$ is inclusion-preserving and surjective by construction.  We claim that $\psi$ is also injective. 
To show this, it suffices to show (as $\phi$ is bijective) that $\psi \circ \phi^{-1}$ is injective.  To this end, let $\tau_1, \tau_2 \in \code$ be such that 
$\tau_1 \cap \rig \neq \emptyset$, 
$\tau_2 \cap \rig \neq \emptyset$, 
and
$\tau_1 \neq \tau_2$.  We must show that 
$\widetilde{\tau}_1 \cap ( \rig \cup \mathcal{I}) \neq \widetilde{\tau}_2 \cap ( \rig \cup \mathcal{I}) $, 
where (for $j=1,2$) $\widetilde{\tau}_j$ denotes the codeword in $\widetilde{\code}$ that corresponds to $\tau_j$: 
\begin{align*}
    \widetilde{\tau}_j ~:=~ \tau_j \cup \{ i_{\omega} \mid i_{\omega} \in \mathcal{I} \textrm{ and } \omega \subseteq \tau_j \}~.
\end{align*}
The assumption $\tau_1 \neq \tau_2$ implies that there exists a neuron $k\in [n]$ such that 
$k \in (\tau_1 \smallsetminus \tau_2)$
or
$k \in (\tau_2 \smallsetminus \tau_1)$.  By relabeling if necessary, we may assume that $k \in (\tau_1 \smallsetminus \tau_2)$.  We consider two cases, based on whether $k \in \rig$.  
If $k\in \rig$, then $k \in \left(\widetilde{\tau}_1 \cap ( \rig \cup \mathcal{I})\right) \smallsetminus \left( \widetilde{\tau}_2 \cap ( \rig \cup \mathcal{I}) \right) $ 
(here, we also use the fact that $k \in (\tau_1 \smallsetminus \tau_2) \subseteq ( \widetilde{\tau}_1 \smallsetminus \widetilde{\tau}_2) $).   
We now consider the remaining case, when $k \notin \rig$.
Then 
$i_{\tau_1} \in  \widetilde{\tau}_1 $ (as $\tau_1 \cap \rig \neq \emptyset$), and
$i_{\tau_1} \notin  \widetilde{\tau}_2 $ (because $\tau_1 \nsubseteq \tau_2$).
Hence, 
as $i_{\tau_1} \in \mathcal{I}$, we have
$i_{\tau_1} \in \left(\widetilde{\tau}_1 \cap ( \rig \cup \mathcal{I})\right) \smallsetminus \left( \widetilde{\tau}_2 \cap ( \rig \cup \mathcal{I}) \right) $.  Thus, our claim holds.

Next, we prove the 
equality asserted in~\eqref{eq:equality-codes}.  
The containment $\subseteq$ is straightforward: 
if $\widetilde{\tau} \cap \rig \neq \emptyset$, 
then $\widetilde{\tau} \cap (\rig \cup \mathcal{I}) \neq \emptyset$.  
As for the reverse containment $\supseteq$, 
redundant neurons $i_{\omega}$ were added only for certain $\omega \in \code$
for which $\omega \cap \rig \neq \emptyset$; hence, 
$\widetilde{\tau} \cap (\rig \cup \mathcal{I}) \neq \emptyset$ implies that
$\widetilde{\tau} \cap \rig \neq \emptyset$.

To summarize, we have proven that the maps in~\eqref{eq:bijections} are inclusion-preserving bijections and that the equality in~\eqref{eq:equality-codes} holds.  
Hence, the codeword-containment graph of $\widetilde{\code}|_{\rig \cup \mathcal{I} } \smallsetminus \{ \emptyset \}$ is isomorphic to the codeword-containment graph of $\{ \tau \in \code \mid \tau \cap \rig \neq \emptyset \}$, which by hypothesis is a path. So, by  Lemma~\ref{lem:case-of-path-prop-with-all-neurons}, 
$(\rig \cup \mathcal{I}, \cup)$ is a rigid structure of 
 $\widetilde{\code}|_{\rig \cup \mathcal{I} }$ and thus (by definition) is also a rigid structure of $\widetilde{\code}$.  So, by repeated application of Lemma~\ref{lem:redundant-rigid}, $(\rig, \cup)$ is a rigid structure of the original code $\code$.
%
\end{proof}

\begin{remark}[Converse of Proposition~\ref{prop:path-rigid} is false] \label{rem:conj-converse-prop-path-rigid}
It is natural to ask about the converse to the proposition -- that is, when $\Gamma$ is a path graph, if $(\rig, \cup) $ is a rigid structure, must the triplewise-intersection condition hold?  
The answer is ``no''.  For instance, when $\Gamma$ is the path $ 1 - 124 - 2 - 234$, we have that $(\{1,2,3,4\}, \cup)$ is rigid for trivial reasons: this code has no closed-convex realization (because the Interval Condition is violated by neuron $4$).  Yet, the triplewise-intersection condition fails: $1 \cap 124 \cap 2$ is empty.
On the other hand, requiring the Interval Condition makes the converse true; this is shown in forthcoming work of 
Luis Gomez, Loan Tran, and Elijah Washington.
\end{remark}

\subsection{Proof of Theorem~\ref{thm:rigid-when-gph-is-cycle}} \label{sec:proof-cycle}
We use Proposition~\ref{prop:path-rigid} to prove the result, stated earlier, asserting the non-closed-convexity of codes for which the codeword-containment graph is a cycle and the triplewise-intersection condition holds.

\begin{proof}[Proof of Theorem~\ref{thm:rigid-when-gph-is-cycle}]
Assume that the codeword-containment graph of $\code \smallsetminus \{ \emptyset \}$ is a cycle $\sigma_1 - \sigma_2 - \dots - \sigma_q - \sigma_1$ (so, $q \geq 3$) 
and that the triplewise intersections $\sigma_i \cap \sigma_{i+1} \cap \sigma_{i+2}$ are nonempty. We first rule out the $q=3$ case.  In this case, the three (nonempty) codewords $\sigma_i$ are linearly ordered, that is, after relabeling if needed, $\sigma_1 \supsetneq \sigma_2 \supsetneq \sigma_3$.  Hence, $\sigma_1 \cap \sigma_2 \cap \sigma_3 = \sigma_3 \neq \emptyset$, which violates the triplewise-intersection condition.

Assume that $q \geq 4$.
Relabel, if necessary, so that $\sigma_1$ is a maximal codeword.  
Hence, $\sigma_1 \supsetneq \sigma_2$.  Also, $\sigma_2 \subsetneq \sigma_3$, because if instead we had $\sigma_2 \supsetneq \sigma_3$, then $\sigma_1 - \sigma_3$ would be an edge of the codeword-containment graph, which is impossible as the graph is a cycle of length $q \geq 4$. 
We claim that
(i) no nonempty codeword of $\code$ is properly contained in $\sigma_2$, and (ii) $\sigma_1$ and $\sigma_3$ are the only codewords that properly contain $\sigma_2$.  Indeed, either were violated, the codeword-containment graph would contain an edge from $\sigma_2$ to some $\sigma_k$, with $k \neq 1,3$, and hence would not be a cycle.
Facts (i) and (ii) will be used below.

Let $\rig:= \sigma_2$ and $\conn:= [n] \smallsetminus \rig$.  The pair $(\rig, \cap) $ is automatically a rigid structure, and we now claim that $(\conn, \cup)$ is also a rigid structure.  To see this, note that no nonempty codeword is properly contained in $\rig=\sigma_2$ and each such codeword contains at least one neuron from $\conn$.  Thus, the codeword-containment graph of $\{ \tau \in \code \mid \tau \cap \conn \neq \emptyset \}$ is the path $\sigma_3 - \sigma_4 - \dots - \sigma_q - \sigma_1$.  
The triplewise-intersection property for this path holds, by hypothesis. 
Hence, Proposition~\ref{prop:path-rigid} implies that $(\conn, \cup) $ is a rigid structure, as desired.

The distinguished subcode arising from $(\rig, \cap) $ and $(\conn, \cup)$ is $\{ \tau \in \code \mid \rig=\sigma_2 \subseteq \tau ~{\rm and}~ \tau \cap \conn \neq \emptyset \} = \{\sigma_1,\sigma_3\}$, and the resulting codeword-containment graph is disconnected ($\sigma_1 \nsubseteq \sigma_3$ and $\sigma_1 \nsupseteq \sigma_3$).  So, by Theorem~\ref{thm:rigid}, $\code$ is not closed-convex.
\end{proof}

\begin{remark} \label{rem:many-choices-of-r-and-c}
From the proof of Theorem~\ref{thm:rigid-when-gph-is-cycle}, 
we see that 
codes satisfying the hypotheses of that theorem
have many choices of rigid structures $(\rig, \cap)$ and $([n] \smallsetminus \rig, \cup)$ that can be used to prove non-closed-convexity.  For instance, for the code $C15$ (from Example~\ref{ex:cycle}), $\rig$ can be chosen to be any of the following subsets of $\{1,2,3,4,5\}$: $12$, $15$, $45$, $34$, and $23$. 
\end{remark}

\subsection{Examples} \label{sec:examples-rigid}
Here we show that four of the five known non-closed-convex codes that are open-convex (those listed in Table~\ref{tab:5-codes}), can be handled by Theorem~\ref{thm:rigid}.  The rigid structures that we use are listed in Table~\ref{tab:5-codes} and some are shown schematically in Figures~\ref{fig:C6-rigid} and~\ref{fig:C10-rigid}.

\begin{table}[ht]
\begin{tabular}{lccc}
\hline 
Code      & $(\rig, \bigcirc)$ & $(\conn, \bigtriangleup)$ & Distinguished subcode     \\
\hline 
$C6$        & $(1235, \cup)$         & 4      & $\{145, 234\}$        \\
$C10$       & $(34, \cup)$ 
                        & 5 & $\{ 135,245 \}$              \\
$C15$       & $(345,\cup)$   & $(12,\cap)$  & $\{ 123, 125 \}$ \\
$\CruzCode$ & $(1456, \cup)$               & $(23, \cap)$  & $\{ 123, 234\}$   \\
\hline
\end{tabular}
\caption{
For each code, we list two rigid structures 
and the resulting distinguished subcode.  
See Proposition~\ref{prop:examples-rigid}.
} \label{tab:5-codes}
\end{table}

\begin{proposition} \label{prop:examples-rigid}
    The codes $C6$, $C10$, $C15$, and $ \CruzCode$ (from Example~\ref{ex:goldrup-phillipson}) satisfy the hypotheses of Theorem~\ref{thm:rigid}, where the rigid structures are those in Table~\ref{tab:5-codes}, and are therefore non-closed-convex.
\end{proposition}

\begin{proof}
For the code $C6$, we saw in Example~\ref{ex:goldphil-C6-rig} that 
$(\rig, \bigcirc ) := (1235, \cup)$ is a rigid structure.  We also consider the rigid structure 
$(\conn, \bigtriangleup):= (4, \cap)$.  The resulting distinguished subcode is 
$C6' = \{ 145, 234 \}$, which has disconnected codeword-containment graph.
So, by Theorem~\ref{thm:rigid}, the code $C6$ is not closed-convex.

Next, we consider the code $C10$.  Applying Proposition~\ref{prop:path-rigid}
to $C10|_{1234}$ and $\rig = \{3,4\}$, the codeword-containment graph $\Gamma$ is the 
path $13 - 134 - 34 - 234 - 24$, and the intersections 
 $13 \cap 134 \cap 34  = 3$, 
  $134 \cap 34 \cap 234  =34$, and 
   $ 34 \cap 234 \cap 24=4$ are nonempty.  We conclude that $(34, \cup) $ is a rigid structure of $C10$.  The other rigid structure we consider is $(5, \cap)$.  The resulting distinguished subcode is $\{135, 245\}$, and so, as above, the code $C10$ is not closed-convex.

Finally, the codes $C15$ and $\CruzCode$ were already analyzed in Example~\ref{ex:cycle} (see also the proof of Theorem~\ref{thm:rigid-when-gph-is-cycle}).
%
%
\end{proof}

    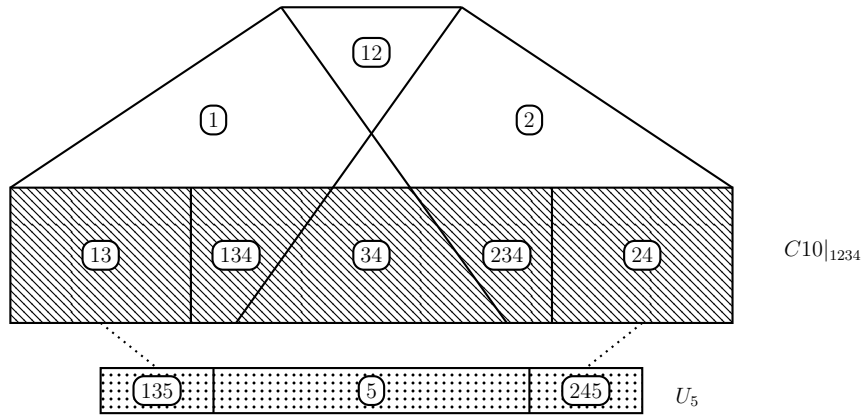
\begin{figure}[ht]
        \centering
         \begin{tikzpicture}[thick,scale=.6,every node/.style={scale=0.7}]
            \begin{scope}
             	\node[label=$C10 | _ {1234} $] at (18,1) {};
                \draw[pattern=north west lines] (0,0) rectangle (16,3){};
                \draw (4,0)--(4,3) {};
                \draw (12,0) -- (12,3) {};
                \draw (6,7) -- (10,7) {};
                \draw (0,3) -- (6,7) {};
                \draw (16,3) -- (10,7) {};
                \draw (5,0)--(10,7){};
                \draw (11,0)--(6,7){};
                \draw(2,1.5) node [fill=white,draw,rounded corners]{13};
                \draw(5,1.5) node [fill=white,draw,rounded corners]{$134$};
                \draw(8,1.5) node [fill=white,draw,rounded corners]{$34$};
                \draw(11,1.5) node [fill=white,draw,rounded corners]{$234$};
                \draw(14,1.5) node [fill=white,draw,rounded corners]{$24$};
                \draw(4.5,4.5) node [fill=white,draw,rounded corners]{$1$};
                \draw(11.5,4.5) node [fill=white,draw,rounded corners]{$2$};
                \draw(8,6) node [fill=white,draw,rounded corners]{$12$};
                
             	\node[label=$U_5 $] at (15,-2.2) {};
                \draw[pattern = dots] (2,-1) rectangle (14,-2){};
                \draw(3.25,-1.5) node [fill=white,draw,rounded corners]{$135$};
                \draw(12.75,-1.5) node [fill=white,draw,rounded corners]{$245$};
                \draw(8,-1.5) node [fill=white,draw,rounded corners]{$5$};
                \draw(4.5,-1)--(4.5,-2){};
                \draw(11.5,-1)--(11.5,-2){};

                \draw[dotted] (2,0) -- (3.25,-1);
                \draw[dotted] (14,0) -- (12.75,-1);
            \end{scope}
        \end{tikzpicture}
         \caption{The code $C10$, with two rigid structures (shaded).  Dotted lines indicate pairs of codewords arising from distinct rigid structures, with one contained in the other.}
                \label{fig:C10-rigid}
    \end{figure}

Next, we give the first infinite family of codes that are open-convex but non-closed-convex.

\begin{theorem} \label{thm:infinite-family-non-closed}
For $n \geq 5$, the following neural code is open-convex but not closed-convex:
\begin{align}
\label{eq:infinite-fam-1}
\ncfamn ~=~ & 
            \{12,~ \mathbf{123},~ 23,~ \mathbf{234},~ 34,~ \dots,~ 
                \mathbf{ (n-2)(n-1)n},~ (n-1)n, \\     
            & \quad \mathbf{ 12(n+1)},~ n+1, ~
            \mathbf{
            (n-1)n(n+1)}, ~ \emptyset \}. \notag
\end{align}
\end{theorem}

\begin{proof}
It is straightforward to check that the open-convex realization of $\mathcal{D}_5$ shown in Figure~\ref{fig:pinwheel-infinte-family} extends to all $\mathcal{D}_n$, for $n \geq 5$.

        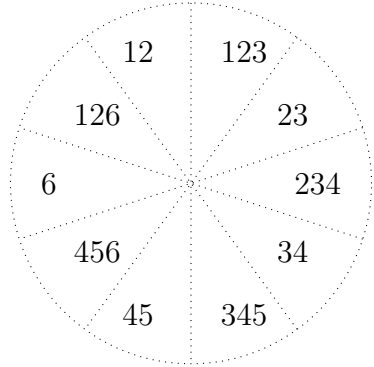
\begin{figure}[ht]
        \centering
    \begin{tikzpicture}[scale=2.4]
    \draw[black,  dotted] (0,1) -- ({sin(5*\z)},{cos(5*\z)});
    \draw[black,  dotted]  ({sin(1*\z)},{cos( 1*\z)}) -- ({sin( 6*\z)},{cos( 6*\z)}) ;
    \draw[black,  dotted]  ({sin(2*\z)},{cos( 2*\z)}) -- ({sin( 7*\z)},{cos( 7*\z)}) ;
    \draw[black,  dotted]  ({sin(3*\z)},{cos( 3*\z)}) -- ({sin( 8*\z)},{cos( 8*\z)}) ;
    \draw[black,  dotted]  ({sin(4*\z)},{cos( 4*\z)}) -- ({sin( 9*\z)},{cos( 9*\z)}) ;
    \draw[black, dotted] circle(1);
    \node[label=left:$234$] at ({sin(2*\z)/2+sin(3*\z)/2},{cos(2*\z)/2+cos(3*\z)/2}){};
    \node[label=above left:$34 $] at ({sin(3*\z)/2+sin(4*\z)/2},{cos(3*\z)/2+cos(4*\z)/2}){};
    \node[label=$ 345$] at ({sin(4*\z)/2+sin(5*\z)/2},{cos(4*\z)/2+cos(5*\z)/2}){};
    \node[label=$45 $] at ({sin(5*\z)/2+sin(6*\z)/2},{cos(5*\z)/2+cos(6*\z)/2}){};
    \node[label=above right: $456 $] at ({sin(6*\z)/2+sin(7*\z)/2},{cos(6*\z)/2+cos(7*\z)/2}){};
    \node[label=right:$ 6$] at ({sin(7*\z)/2+sin(8*\z)/2},{cos(7*\z)/2+cos(8*\z)/2}){};
    \node[label=below right:$126 $] at ({sin(8*\z)/2+sin(9*\z)/2},{cos(8*\z)/2+cos(9*\z)/2}){};
    \node[label=below:$12 $] at ({sin(9*\z)/2+sin(10*\z)/2},{cos(9*\z)/2+cos(10*\z)/2}){};
    \node[label=below:$ 123$] at ({sin(10*\z)/2+sin(1*\z)/2},{cos(10*\z)/2+cos(1*\z)/2}){};
    \node[label=below left:$23 $] at ({sin(1*\z)/2+sin(2*\z)/2},{cos(1*\z)/2+cos(2*\z)/2}){};
    \end{tikzpicture}
        \caption{An open-convex realization of $\mathcal{D}_5$ in $\mathbb{R}^2$.}
        \label{fig:pinwheel-infinte-family}
        \end{figure}

We apply  Proposition~\ref{prop:path-rigid}
to the code $\mathcal{D}_n$ and $\rig = \{1,2,\dots,n\}$.  The resulting graph $\Gamma$ is the 
path on the codewords -- in that order, from $12$ to $(n-1)n$ -- shown in line~\eqref{eq:infinite-fam-1}.  
The consecutive triplewise intersections are easily seen to be nonempty, and so $([n], \cup) $ is a rigid structure of the code.  
See Figure~\ref{fig:D_n-rigid}.
The other rigid structure we consider is $(n+1, \cap)$.  The resulting distinguished subcode is $\{ 12(n+1),
            (n-1)n(n+1) \}$, and so by Theorem~\ref{thm:rigid},  $\mathcal{D}_n$ is not closed-convex.  
\end{proof}

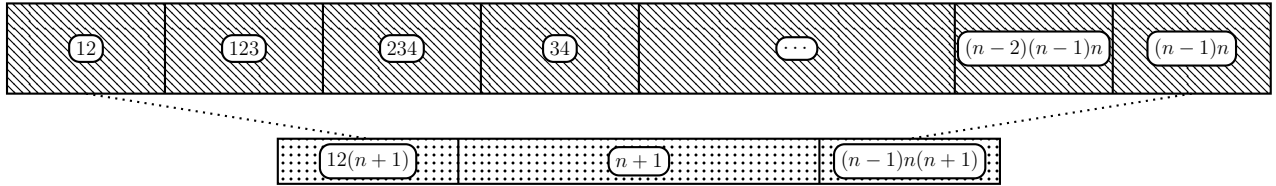
\begin{figure}[ht]
        \centering
        \begin{tikzpicture}[thick,scale=.6, every node/.style={scale=0.65}]
            \begin{scope}
                \draw[pattern=north west lines] (0,0) rectangle (28,2){};
                \draw (3.5,0)--(3.5,2) {};
                \draw (7,0)--(7,2) {};
                \draw (10.5,0)--(10.5,2) {};
                \draw (14,0)--(14,2) {};
                \draw (21,0)--(21,2) {};
                \draw (24.5,0)--(24.5,2) {};
                \draw(1.75,1) node [fill=white,draw,rounded corners]{$12$};
                \draw(5.25,1) node [fill=white,draw,rounded corners]{$123$};
                \draw(8.75,1) node [fill=white,draw,rounded corners]{$234$};
                \draw(12.25,1) node [fill=white,draw,rounded corners]{$34$};
                \draw(17.5,1) node [fill=white,draw,rounded corners]{$\cdots$};
                \draw(22.75,1) node [fill=white,draw,rounded corners]{$(n-2)(n-1)n$};
                \draw(26.25,1) node [fill=white,draw,rounded corners]{$(n-1)n$};
                %
                \draw[pattern= dots] (6,-1) rectangle (22,-2){};
                \draw(8,-1.5) node [fill=white,draw,rounded corners]{$12(n+1)$};
                \draw(14,-1.5) node [fill=white,draw,rounded corners]{$n+1$};
                \draw(20,-1.5) node [fill=white,draw,rounded corners]{$(n-1)n(n+1)$};
                \draw(10,-1)--(10,-2){};
                \draw(18,-1)--(18,-2){};
                
                \draw[dotted] (1.75,0) -- (8,-1);
                \draw[dotted] (26.25,0) -- (20, -1);
            \end{scope}
        \end{tikzpicture}
        \caption{ The code $\mathcal{D}_n$, viewed as the union of two rigid structures. Dotted lines indicate pairs of codewords arising from distinct rigid structures, with one contained in the other.}
                \label{fig:D_n-rigid}
    \end{figure}

\begin{remark} \label{rem:partial-order}
Recall that closed-convexity is ``minor-closed'' with respect to the partial order defined in~\cite{jeffs2020morphisms} (see \cite[Proposition~9.3]{jeffs2019embedding}).  
It would therefore be interesting to see whether,
with respect to this partial order, 
the codes $\ncfamn$ are minimally non-closed-convex.
\end{remark}

\begin{remark} \label{rem:scan-codes}
Another direction for future work is to use the results in this section to scan codes on 6 neurons for non-closed-convexity, like was done for open-convexity in the work of Ruys de Perez, Matusevich, and Shiu~\cite{wheels}.
\end{remark}

The next example is a non-closed-convex code that, as far as we know, does not satisfy the hypotheses of Theorem~\ref{thm:rigid}, but may indicate a way to extend our results.

        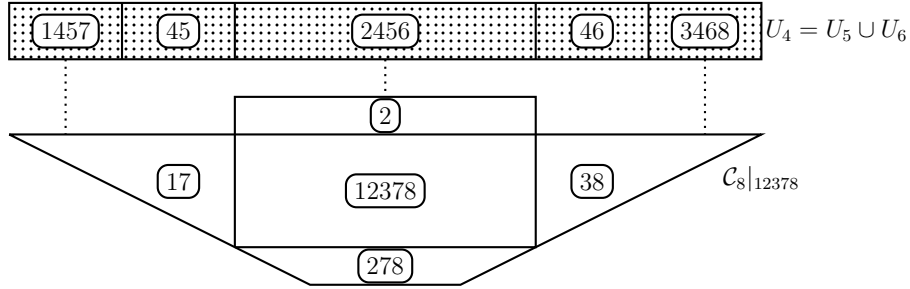
\begin{figure}[ht]
        \centering
        \begin{tikzpicture}[thick,scale=.5,every node/.style={scale=0.8}]
            \begin{scope}
             	\node[label={$U_4= U_5 \cup U_6$}] at (22,6) {};
             	\node[label={$\code_8 |_{12378}$}] at (20,2) {};
                \draw(0,4)--(8,0)--(12,0)--(20,4)--(0,4){}; 
                \draw (6,1)--(14,1)--(14,5)--(6,5)--(6,1) {};
                \draw[pattern = dots] (0,6) rectangle (20,7.5){};
                \draw(3,6)--(3,7.5){};
                \draw(6,6)--(6,7.5){};
                \draw(14,6)--(14,7.5){};
                \draw(17,6)--(17,7.5){};
                \draw(4.5,2.75) node [fill=white,draw,rounded corners]{$17$};
                \draw(10,2.5) node [fill=white,draw,rounded corners]{$12378$};
                \draw(15.5,2.75) node [fill=white,draw,rounded corners]{$38$};
                \draw(10,.5) node [fill=white,draw,rounded corners]{$278$};
                \draw(10,4.5) node [fill=white,draw,rounded corners]{$2$};
                
                \draw(1.5,6.75) node [fill=white,draw,rounded corners]{$1457$};
                \draw(10,6.75) node [fill=white,draw,rounded corners]{$2456$};
                \draw(4.5,6.75) node [fill=white,draw,rounded corners]{$45$};
                \draw(15.5,6.75) node [fill=white,draw,rounded corners]{$46$};
                \draw(18.5,6.75) node [fill=white,draw,rounded corners]{$3468$};
                
                \draw[dotted] (1.5,6)-- (1.5,4);
                \draw[dotted] (18.5,6)-- (18.5,4);
                \draw[dotted] (10,6)-- (10,5);
            \end{scope}
        \end{tikzpicture}
         \caption{The code $\code_8$.  The shaded region at the top is rigid.  Dotted lines indicate pairs of codewords, one contained in the other, arising from distinct regions.}
                \label{fig:non-monotone-code-rigid}
    \end{figure}

\begin{example} \label{ex:non-monotone-code-rigid}
We revisit the non-closed-convex code 
$\code_8$ from Example~\ref{ex:non-monotone-code}.
It has a rigid structure (shown in Figure~\ref{fig:non-monotone-code-rigid}). Informally, 
this structure must ``bend'' to intersect the required atoms in the rest of the code, and so 
$\code_8$ is non-closed-convex. 
It would be interesting in the future to generalize our results to accommodate this code.
\end{example}

We end this section by showing how rigid structures can be used to preclude closed-convexity, even when the results in this section do not directly apply (or are difficult to apply).  
The idea, which we illustrate in the following example, is to replace a rigid structure (which is necessarily convex) by a new neuron.

\begin{example} \label{ex:simplifaction-through-rigidity}
Consider the following code on 8 neurons: 
\begin{align*}
        \mathcal{D} ~&=~ \{\mathbf{12467},\mathbf{2678},\mathbf{123},
        \mathbf{138},\mathbf{345},
        \mathbf{456},
         \\
        & \quad \quad 
        1246,1267,
        2467,267,
        467,12,13,45,46,67,
        3,8, \emptyset \}~.
\end{align*}
By applying  Lemma~\ref{lem:case-of-path-prop-with-all-neurons}
to the restricted code $\mathcal{D}|_{4567}$, we conclude that $(4567,\cup)$ is a rigid structure of 
$\mathcal{D}|_{4567}$ and thus (by definition) is also a rigid structure of
$\mathcal{D}$, shown with shading here:

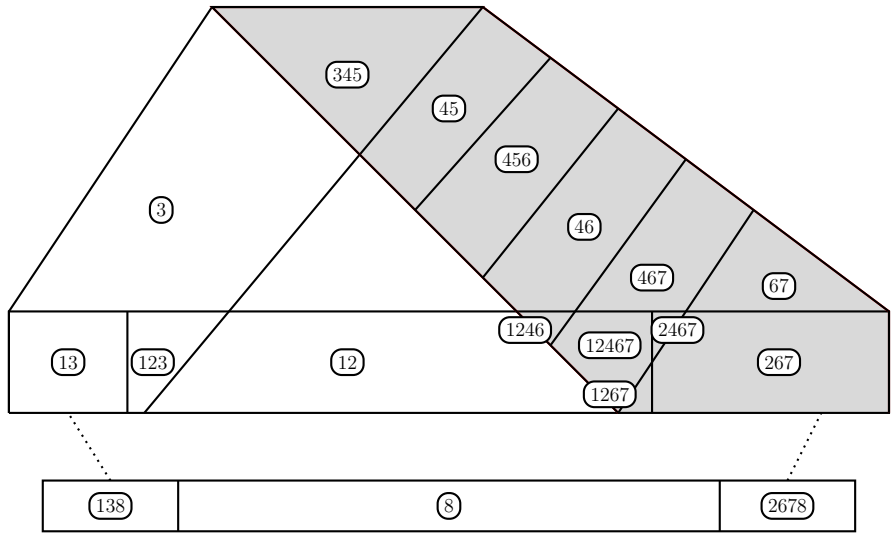
\begin{center}
        \begin{tikzpicture}[thick,scale=.45,every node/.style={scale=0.625}]
            \begin{scope}
                \draw[] (0,0)--(26,0)--(26,3)--(0,3)--(0,0){}; 
                \draw (0,0)--(0,3)--(6,12)--(14,12)--(4,0) {};
                \draw [red,  fill=black, fill opacity=0.15] (26,0)--(26,3)--(14,12)--(6,12)--(18,0) {};
                \draw [] (26,0)--(26,3)--(14,12)--(6,12)--(18,0) {}; 
                \draw[] (1,-3.5) rectangle (25,-2){}; 
                \draw(3.5,0)--(3.5,3){};
                \draw(19,0)--(19,3){};
                \draw(12,6)--(16,10.5){};
                \draw(14,4)--(18,9){};
                \draw(16,2)--(20,7.5){};
                \draw(18,0)--(22,6){};
                \draw(5,-3.5)--(5,-2){};
                \draw(21,-3.5)--(21,-2){};
                \draw(1.75,1.5) node [fill=white,draw,rounded corners]{$13$};
                \draw(4.25,1.5) node [fill=white,draw,rounded corners]{$123$};
                \draw(10,1.5) node [fill=white,draw,rounded corners]{$12$};
                \draw(4.5,6) node [fill=white,draw,rounded corners]{$3$};
                \draw(10,10) node [fill=white,draw,rounded corners]{$345$};
                \draw(13,9) node [fill=white,draw,rounded corners]{$45$};
                \draw(15,7.5) node [fill=white,draw,rounded corners]{$456$};
                \draw(17,5.5) node [fill=white,draw,rounded corners]{$46$};
                \draw(19,4) node [fill=white,draw,rounded corners]{$467$};
                \draw(22.75,3.75) node [fill=white,draw,rounded corners]{$67$};
                \draw(15.25,2.45) node [fill=white,draw,rounded corners]{$1246$};
                \draw(17.75,2) node [fill=white,draw,rounded corners]{$12467$};
                \draw(17.75,.55) node [fill=white,draw,rounded corners]{$1267$};
                \draw(19.75,2.45) node [fill=white,draw,rounded corners]{$2467$};
                \draw(22.75,1.5) node [fill=white,draw,rounded corners]{$267$};
                \draw(13,-2.75) node [fill=white,draw,rounded corners]{$8$};
                \draw(3,-2.75) node [fill=white,draw,rounded corners]{$138$};
                \draw(23,-2.75) node [fill=white,draw,rounded corners]{$2678$};
                \draw[dotted] (3,-2)-- (1.75,0);
                \draw[dotted] (23,-2)-- (24,0);
            \end{scope}
        \end{tikzpicture}
\end{center}

Next, as 
$U_4 \cup U_5 \cup U_6 \cup U_7$ 
is convex (in every closed-convex realization of $\mathcal{D}$), 
we can introduce a new neuron $B$ for which $U_B= U_4 \cup U_5 \cup U_6 \cup U_7$. 
More precisely, let $\widetilde{\mathcal{D}}$ be the code obtained from $\mathcal{D}$ by adding neuron $B$ to every codeword that contains at least one neuron from the set $\{4,5,6,7\}$.  
It follows that $\mathcal{D}$ and $\widetilde{\mathcal{D}}$ are equivalent (from the point of view of analyzing closed-convexity). The following restricted code
is (up to relabeling neurons) the code $C10$:
\begin{align*}
\widetilde{\mathcal D}|_{1238B} ~=~ \{ \mathbf{123}, \mathbf{12B}, \mathbf{138}, \mathbf{28B}, \mathbf{3B}, 12,13,2B,3,8,B, \emptyset \}~.
\end{align*}
Results in this section (and also those in the next section) show that $C10$ is non-closed-convex, and thus so is $\widetilde{\mathcal{D}}$, and hence $\mathcal{D}$ as well.
\end{example}

\begin{remark}
We have already used rigid structures to show that the code $C6$ is non-closed-convex
(Proposition~\ref{prop:examples-rigid}). 
Another way to see this uses the ``new neuron'' idea in the previous example.  Namely, as $(1235, \cup) $ is a rigid structure, we let $R$ denote a new neuron with $U_R=U_1 \cup U_2 \cup U_3 \cup U_5$.  Then, one can check that this new code has (at $R$) a ``local obstruction'' to convexity (which precludes  closed-convexity), as in~\cite{what-makes,no-go}.
\end{remark}

\section{Precluding closed-convexity using RF relationships} \label{sec:criterion-RF}
In this section, we introduce a second criterion for precluding closed-convexity, which, unlike 
the one in the previous section (Theorem~\ref{thm:rigid}), 
can be checked directly from a neural code or its neural ideal (Theorem~\ref{thm:criterion-non-closed-convex-via-RF}).  
The conditions in our criterion, which are listed in Table~\ref{tab:criterion}, capture the RF relationships that lead to non-closed-convexity in some codes in the literature.
Accordingly, our proof unifies arguments made in~\cite{goldrup2020classification}. 
Additionally, we obtain a criterion that can be checked from the canonical form (Corollary~\ref{cor:criterion-CF}) and therefore is the first algebraic signature for non-closed-convexity (besides signatures of local obstructions~\cite{curto2018algebraic}).



\begin{table}[ht]
\begin{tabular}{lcc}
\hline
Property of neural code                                    & RF relationship                    & Neural ideal                            \\
\hline
Some codeword contains $\{i,j,k\}$                           & $U_{ijk} \neq \emptyset$            & $x_i x_j x_k \notin J_{\code}$          \\
Some codeword contains $\{i,j,m\}$                           & $U_{ijm} \neq \emptyset$            & $x_i x_j x_m \notin J_{\code}$          \\
Some codeword contains $\{k,\ell, m\}$                       & $U_{k\ell m} \neq \emptyset$        & $x_k x_{\ell} x_m \notin J_{\code}$     \\
No codeword contains $\{i,j,k, \ell \}$                           & $U_{ijk \ell } = \emptyset$              & $x_i x_j x_k x_{\ell} \in J_{\code}$      \\
No codeword contains $\{i,j,k,m\}$                           & $U_{ijkm} = \emptyset$              & $x_i x_j x_k x_m \in J_{\code}$      \\
Each codeword containing $k$ also contains $j$ or $\ell$ & $U_k \subseteq (U_j \cup U_{\ell})$ & $x_k (x_j+1)(x_{\ell}+1) \in J_{\code}$ \\
Each codeword containing $j$ also contains $i$ or $k$    & $U_j \subseteq (U_i \cup U_{k})$    & $x_j (x_i+1)(x_{k}+1) \in J_{\code}$   \\
\hline
\end{tabular}
\caption{Properties of a neural code $\code$, the equivalent receptive field relationships, and the equivalent conditions on the neural ideal $J_{\code}$.   \label{tab:criterion}}
\end{table}

\begin{theorem}[Criterion for precluding closed-convexity] \label{thm:criterion-non-closed-convex-via-RF}
Let $\code$ be a code on $n$ neurons. 
Let $i,j,k,l,m \in [n]$.  If $\code$ satisfies the properties in the first column of Table~\ref{tab:criterion} or, equivalently, the neural ideal $J_{\code}$ satisfies the properties in the third column, 
then $\code$
 is not closed-convex.
\end{theorem}

\begin{proof}
The equivalence of columns two and three of Table~\ref{tab:criterion} 
is due to Lemma~\ref{lem:receptivefields}, and the equivalence of columns one and two is straightforward.

Assume for contradiction that $\{U_1,U_2,\dots, U_n\}$ is a closed-convex realization of $\code$.  By intersecting each $U_i$ with a sufficiently large closed ball (the same ball for all $i$), we may assume that each $U_i$ is compact.

The first four RF relationships in Table~\ref{tab:criterion} imply that, for the three sets $U_{ij}, U_k, U_m$, all pairwise intersections are nonempty but the triplewise intersection 
$U_{ij} \cap U_k \cap U_m = U_{ijkm}$
is empty.  We depict this schematically here: 


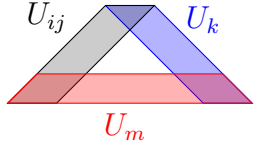
\begin{center}
    \begin{tikzpicture}[scale=1.3]
    \draw[blue,  fill=blue, fill opacity=0.3] 
	(0,0) -- (0.5,0) -- (1.5,-1) -- (1,-1) -- (0,0)  ;
    \draw[black,  fill=black, fill opacity=0.2] 
	(0,0) 
	-- (0.5,0) 
	-- (-0.5,-1) -- (-1,-1) -- (0,0) ;
    \draw[red,  fill=red, fill opacity=0.3] 
	 (-1,-1) -- (1.5,-1) -- (1.2,-0.7) -- (-0.7, -0.7) -- (-1, -1) ;
	\node[label=${\color{blue} U_{k}}$] at (1, -0.5) {};
	\node[label=${\color{red} U_{m}}$] at (0.2,-1.6) {};
	\node[label=$U_{ij}$] at (-0.6, -0.5) {};
    \end{tikzpicture}
\end{center}

We can therefore pick distinct points $y_{ijm} \in U_{ijm}$ and $y_{k \ell m} \in U_{k \ell m}$ (recall from Table~\ref{tab:criterion} that $U_{k \ell m} \neq \emptyset$). 
Let $L_m$ be the line segment joining the two points.  We have $L_m \subseteq U_m$, because both endpoints are in the convex set $U_m $.  Next, $L_m$ and $U_{ijk}$ are compact and disjoint sets (recall that $U_{ijkm} = \emptyset$), so there exists a point $y_{ijk} \in U_{ijk}$ that achieves the minimum distance from $U_{ijk}$ to $L_m$ (and this distance is positive).  

Letting $L_{ij}$ be the line segment from $y_{ijm}$ to $y_{ijk}$ (so, $L_{ij} \subseteq U_{ij}$ by convexity), and $L_{k}$ the line segment from $y_{ijk}$ to $y_{k \ell m}$ (so, $L_k \subseteq U_k$ by convexity), we have the following triangle:

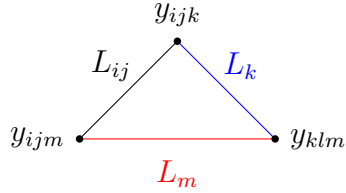
\begin{center}
    \begin{tikzpicture}[scale=1.3]
	\tikzstyle{vertex}=[circle,fill=black, inner sep = 1pt]
	\coordinate (A) at (-.5,-.5);
	\coordinate (B) at (.5,-.5);
	\coordinate (C) at (0,-1);
	\draw[black] (-1,-1)--(0,0);
	\draw[blue] (0,0)--(1,-1);
	\draw[red] (1,-1)--(-1,-1) ;
	\node[vertex, label=$y_{ijk}$] (nodeijk) at (0,0) {};
	\node[vertex, label=left:$y_{ijm}$] (node145) at (-1,-1) {};
	\node[vertex, label=right:$y_{klm}$] (node345) at (1,-1) {};
	\node[label=left:$L_{ij}$] (node125) at (-.25,-.25) {};
	\node[label=right:${\color{blue} L_k}$] (node234) at (.25,-.25) {};
	\node[label=below:${\color{red} L_m}$] at (0,-1){};
\end{tikzpicture}
\end{center}

Recall again that $U_{ijkm} = \emptyset$, so the line segment $L_{ij} \subseteq U_{ij}$ must exit $U_m$ before entering $U_k$.  Hence, there exists a point $y_{ij}$ in the relative interior of the line segment $L_{ij}$ such that $y_{ij} \in U_{ij} \smallsetminus (U_m \cup U_k)$.

Next,  $L_k \subseteq U_k $ and 
(from Table~\ref{tab:criterion}) 
$U_k \subseteq (U_j \cup U_{\ell})$.  Also, the endpoints of $L_k$ are $y_{ijk} \in U_j$ and $y_{k \ell m} \in U_{\ell}$.  So, $L_k$ is a connected set covered by the nonempty closed sets 
$L_k \cap U_j$ and $L_k \cap U_{\ell}$.  
We conclude that there exists a point $y_{j k \ell}$ on $L_k$ that is in $U_{jk \ell}$.  Additionally,
$y_{j k \ell} \notin U_i$ (and in particular $y_{j k \ell} \neq y_{ijk}$) because $U_{ijk\ell} = \emptyset$ (see Table~\ref{tab:criterion}).

We conclude that the line segment from $y_{ij}$ to $y_{jk \ell}$, which we denote by $L_j$, is contained in $U_j$ and (except for the endpoints) lies in the interior of the triangle.
From Table~\ref{tab:criterion}, we have $U_j \subseteq (U_i \cup U_k)$.  Also, recall that $y_{ij} \in (U_i \smallsetminus U_k)$ and $y_{jk \ell} \in (U_k \smallsetminus U_i)$.  
Thus, there is a point $y'_{ijk}$ in the relative interior of $L_j$ (and thus in the interior of the triangle) and also in $U_{ijk}$.  We depict this here:

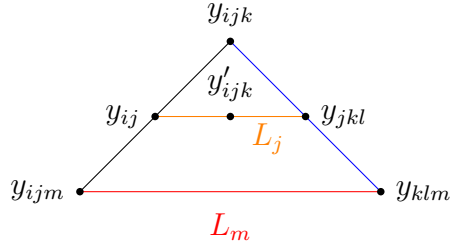
\begin{center}
    \begin{tikzpicture}[scale=2]
	\tikzstyle{vertex}=[circle,fill=black, inner sep = 1pt]
	\coordinate (A) at (-.5,-.5);
	\coordinate (B) at (.5,-.5);
	\coordinate (C) at (0,-1);
	\draw[black] (-1,-1)--(0,0);
	\draw[blue] (0,0)--(1,-1);
	\draw[red] (1,-1)--(-1,-1) ;
	\draw[orange] (A)--(B);
	\node[vertex, label=$y_{ijk}$] (nodeijk) at (0,0) {};
	\node[vertex, label=left:$y_{ijm}$] (node145) at (-1,-1) {};
	\node[vertex, label=right:$y_{klm}$] (node345) at (1,-1) {};
	\node[vertex, label=left:$y_{ij}$] (node125) at (A) {};
	\node[vertex, label=right:$y_{jkl}$] (node234) at (B) {};
	\node[vertex, label=$y'_{ijk}$] (node234) at (0,-.5) {};
	\node[label=below:${\color{red} L_m}$] at (0,-1){};
	\node[label=below:${\color{orange} L_{j}}$] at (0.25,-0.4){};
\end{tikzpicture}
\end{center}

Thus, $y'_{ijk}$ is a point in $U_{ijk}$ that is closer to $L_m$ than $y_{ijk}$ is, which is a contradiction.
\end{proof}


\begin{table}[ht]
\begin{tabular}{lc}
\hline
Code  & $(i,j,k,\ell, m)$ \\
\hline
$C6 =~ \{ {\bf 123}, {\bf 125}, {\bf 145}, {\bf 234},
            12, 15, 23,
            4, 
            \emptyset\}$ 
            & $(3,2,1,5,4)$     \\
$C10 = \{ {\bf 134}, {\bf 135},  {\bf 234}, {\bf 245},
            {\bf 12}, 13, 24, 34,
            1,2,5,
            \emptyset \}$
            & $(1,3,4,2,5)$     \\
$C15 = \{ {\bf 123}, {\bf 125},
            {\bf 145} ,{\bf  234}, {\bf 345},
            12,15,23,34,45, \emptyset \}$ 
            & $(1,2,3,4,5)$     \\
\hline
\end{tabular}
\caption{Three codes and indices $(i,j, k, \ell, m)$ for which the codes satisfy the hypotheses of Theorem~\ref{thm:criterion-non-closed-convex-via-RF}. \label{tab:ijklm}}
\end{table}

\begin{example} \label{ex:3-codes-via-RF-criterion}
Theorem~\ref{thm:criterion-non-closed-convex-via-RF} applies to the codes $C6$, $C10$, and $C15$ from Example~\ref{ex:goldrup-phillipson}, where the values of $(i,j, k, \ell, m)$ are as listed in Table~\ref{tab:ijklm}.  Hence, these codes are non-closed-convex, and our result unifies the proofs of non-closed-convexity due to Goldrup and Phillipson, who first analyzed these codes~\cite{goldrup2020classification}. See also another such proof for $C15$, which uses ``order-forcing'', in~\cite[Example 2.13]{order-forcing}.
\end{example}

\begin{remark}[Minimum-distance arguments] \label{rem:cruz-criterion-RF}
Like the earlier proof of Lemma~\ref{lem:hinge}, the proof of Theorem~\ref{thm:criterion-non-closed-convex-via-RF} relies on minimum-distance arguments.  Indeed, our proof is similar to that of Cruz {\em et al.}~\cite[Lemma 2.9]{cruz2019open}
for the code $\CruzCode$
 from Example~\ref{ex:goldrup-phillipson}, but
Theorem~\ref{thm:criterion-non-closed-convex-via-RF} does not apply 
to that code. 
It may be interesting in the future to prove a result that is similar to Theorem~\ref{thm:criterion-non-closed-convex-via-RF} and does apply to $\CruzCode$.
\end{remark}

\begin{remark}[Adding codewords that preserve non-closed-convexity] \label{rem:add-codewords}
For a non-closed-convex code $\code$, Theorem~\ref{thm:criterion-non-closed-convex-via-RF} can be used to generate a list of non-maximal codewords that, when added to $\code$, results in a code that is still non-closed-convex.  For instance, any subset of the following codewords can be added to $C6$ without making the code become closed-convex: 
$3, 34, 45, 5$.
\end{remark}

\begin{remark}[Replacing $i,j,k,\ell,m$ by sets]
Theorem~\ref{thm:criterion-non-closed-convex-via-RF} allows the following generalization: the properties of neural codes listed in Table~\ref{tab:criterion} involving individual neurons $i,j,k,\ell,m$ can be replaced by nonempty sets of neurons $I,J,K,L,M$ (e.g., ``Some codeword contains $\{i,j,k\}$'' becomes ``Some codeword contains $I\cup J \cup K$'').  
The proof of Theorem~\ref{thm:criterion-non-closed-convex-via-RF} readily accommodates this extension.
An alternate proof is through Jeffs's theory of morphisms of codes~\cite{jeffs2020morphisms}. 
Two key ingredients for this proof are as follows: codes that have a surjective morphism to a non-closed-convex code are non-closed-convex~\cite[Remark~1.8]{jeffs2020morphisms}), and the properties of neural codes in Table~\ref{tab:criterion} can be restated in terms of what Jeffs calls ``trunks''.
\end{remark}

We expect that we can not remove any of the hypotheses of  Theorem~\ref{thm:criterion-non-closed-convex-via-RF} (the seven properties listed in Table~\ref{tab:criterion}).  The following example shows what can happen with one missing hypothesis.

\begin{example} \label{ex:remove-1-hypothesis-from-RF-criterion}
Consider the following code:
$\code = \{ {\bf 123}, {\bf 124}, {\bf 235}, {\bf 45}, 12,14,23,35,4,5,\emptyset \}$.  
This code is closed-convex 
(this follows from results in~\cite{cruz2019open,goldrup2020classification}), but was erroneously asserted to be non-closed-convex in a preliminary version of~\cite{non-monotonicity}.  It is easy to check that, when $(i,j,k,\ell,m)=(1,2,3,5,4)$, the code $\code$ satisfies all properties listed in Table~\ref{tab:criterion} except the one in the third row.
\end{example}

Theorem~\ref{thm:criterion-non-closed-convex-via-RF} provides a sufficient condition for being non-closed-convex.  As noted above, the code $\CruzCode$ shows that this condition is not necessary, even for codes without local obstructions. Another such code is in the following result.

\begin{proposition}\label{prop:ex:c8-does-not-stsf}
The code $\mathcal C_8 = 
    \{
     \mathbf{12378}, 
    \mathbf{1457}, \mathbf{2456}, \mathbf{3468},
    278, 
    17, 38, 45, 46, 
    2, 
    \emptyset
\} $ from Example~\ref{ex:goldrup-phillipson} 
is non-closed-convex but fails to satisfy the hypotheses of Theorem~\ref{thm:criterion-non-closed-convex-via-RF}. 
\end{proposition}

\begin{proof}
As noted earlier, non-closed-convexity was proven in~\cite{non-monotonicity}.

It will be useful in our proof to examine the code's canonical form, which is as follows:  
\begin{align*}
&\{ x_1(x_7 + 1), ~ x_3(x_8 + 1),~ x_3x_5,~ (x_4 + 1)x_5,~ x_5x_8,~ x_1x_6,~\\
&(x_4 + 1)x_6,~ x_6x_7,~ x_4(x_5 + 1)(x_6 + 1),~ (x_1 + 1)(x_2 + 1)x_7,~ \\
&(x_1 + 1)x_7(x_8 + 1),~ x_2x_7(x_8 + 1),~ x_2(x_7 + 1)x_8,~ (x_2
+1)(x_3 + 1)x_8,~\\
&(x_3 + 1)(x_7 + 1)x_8,~ (x_2 + 1)x_7x_8,~ x_1x_4(x_5 + 1),~ (x_1 + 1)x_4x_7,~\\
&x_4(x_5 + 1)x_7,~ (x_1 + 1)x_5x_7,~ x_2x_3(x_5 + 1),~ x_2x_4(x_6 + 1),~ x_2x_4x_7,~\\
&x_2x_5(x_6 + 1),~ x_2x_5x_7,~ x_2(x_5 + 1)x_6,~ (x_2 + 1)x_5x_6,~ x_3x_4(x_6 + 1),~\\
&x_2x_4x_8,~ (x_3 + 1)x_4x_8,~ x_4(x_6 + 1)x_8,~ x_4x_7x_8,~ x_2x_6x_8,~ (x_3 + 1)x_6x_8,~\\
&x_1x_2(x_3 + 1),~ x_1x_2x_4,~ x_1x_2x_5,~ x_1x_2(x_8 + 1),~ x_1(x_2 + 1)x_3,~ x_1x_3x_4,~\\
&(x_1 + 1)x_2x_3,~ x_2x_3x_4,~ x_2x_3x_6,~ x_2x_3(x_7 + 1),~ (x_1 + 1)x_3x_7,~\\
&(x_2 + 1)x_3x_7,~ x_3x_4x_7,~ x_1(x_2 + 1)x_8,~ x_1(x_3 + 1)x_8,~ x_1x_4x_8\}
\end{align*}
Recall from Lemma~\ref{lem:receptivefields} that a pseudomonomial in the canonical form, such as $x_1(x_7 + 1)$, indicates a minimal RF relationship, such as $U_1 \subseteq U_7$.

We first note that $\code_8$ is symmetric under switching the neurons $1,5,7$ with, respectively, $3,6,8$.  We must show that no 
neurons play the roles of $i,j,k,l,$ and $m$
in Theorem~\ref{thm:criterion-non-closed-convex-via-RF}.  We focus on the conditions in rows 1, 2, and 5 in Table~\ref{tab:ijklm}, which are together equivalent to the existence of maximal codewords $\sigma_1$ and $\sigma_2$ that 
(i)
contain (respectively) $ijk$ and $ijm$ and 
(ii) are such that $\sigma_1 \neq \sigma_2$ 
(so that no codeword contains $ijkm$).  
As $\{i,j\} \subseteq (\sigma_1 \cap \sigma_2)$, we restrict our attention to pairs of maximal codewords of $\code_8$ whose intersection has size at least 2.  Up to the symmetry mentioned above, there are only two options:
\begin{itemize}
    \item {\bf Case A}: $\{i,j\} = \{4,5\} = 1457 \cap 2456 =\sigma_1 \cap \sigma_2 $ 
    (in which case, $m \in \{2,6\}$ and $k \in \{1,7\}$, or vice-versa), or
    \item {\bf Case B}: $\{i,j\} = \{3,8\} =12378 \cap 3468 =\sigma_1 \cap \sigma_2 $
    (in which case, $m \in \{1,2,7\}$ and $k \in \{4,6\}$, or vice-versa).
\end{itemize}
Before considering the two cases, we first claim that $k \neq 2$.  Indeed, $2 \in \code_8$ and so there is no RF relationship with $k=2$ of the following form 
required by row~6 of Table~\ref{tab:ijklm}:
\begin{align} \label{eq:row-6-of-tab-ijklm}
    U_k \subseteq (U_j \cup U_{\ell})~.
\end{align}

We begin with Case A.  We claim that $k\neq 6$ and $m \neq 6$.  Indeed, the condition in row 3 of Table~\ref{tab:ijklm} requires that some codeword contain $k \ell m$, and there is no maximal codeword containing $16$ or $67$.   Therefore, it must be that $k \in \{1,7\}$ and $m=2$.

We first consider $k=1$.  The only minimal RF relationship of the form $U_1 \subseteq \cup_{p \in \tau} U_p$ is $U_1 \subseteq U_7$
(this can be checked directly or from the canonical form),
so the required condition~\eqref{eq:row-6-of-tab-ijklm} 
(with $j \in \{4,5\}$) implies that $\ell = 7$.  Thus, $ijk \ell = \{1,4,5,7\}$, which is itself a maximal codeword. 
Thus, the condition in row~4 of Table~\ref{tab:ijklm} fails.

The remaining subcase is when $k=7$.  Only two minimal RF relationship are of the form 
$U_7 \subseteq \cup_{p \in \tau}U_p$, namely, $U_7 \subseteq U_1 \cup U_8$ and $U_7 \subseteq U_1 \cup U_2$.  So, as $j \in \{4,5\}$, there is no RF relationship with $k=7$ of the form in~\eqref{eq:row-6-of-tab-ijklm}.  So, this subcase also does not fulfill the hypotheses of the theorem.

We turn now to Case~B.  The possible values for $k$ are $1,4,6,7$.  For $k=1$, there is a single minimal RF relationship of the form $U_1 \subseteq \cup_{p \in \tau}U_p$, namely, $U_1 \subseteq U_7$.  So, we can again use~\eqref{eq:row-6-of-tab-ijklm} (and the assumption $\{i,j\} = \{3,8\}$) to see that $\ell=7$.  Thus, $ijkl = \{1,3,7,8\}$, which is 
contained in the maximal codeword $12378$ and so contradicts the row-4 condition.

For $k=4$, there is one minimal RF relationship of the form $U_4 \subseteq \cup_{p \in \tau}U_p$, namely, $U_4 \subseteq U_5 \cup U_6$.  This relationship does not involve $j \in \{3,8\}$ and so  condition~\eqref{eq:row-6-of-tab-ijklm} does not hold.

For $k=6$, 
only one minimal RF relationship has the form $U_6 \subseteq \cup_{p \in \tau}U_p$, namely, $U_6 \subseteq U_4$.  Hence, $\ell=4$ and so $ijkl = \{3,4,6,8\}$ is a maximal codeword, again violating 
the row-4 condition.

Finally, consider
$k=7$.  Only two RF relationships have the form $U_7 \subseteq \cup_{p \in \tau}U_p$, namely, 
$U_7 \subseteq U_1 \cup U_2$ and 
$U_7 \subseteq U_1 \cup U_8$.  It must therefore be that $j=8$ and $\ell=1$.  So, $ijk \ell = \{1,3,7,8\} \subseteq \{1,2,3,7,8\}$, again contradicting the row-4 condition.  Our proof is now complete.
\end{proof}



%

\begin{corollary}[Precluding closed-convexity using the canonical form] \label{cor:criterion-CF}
Let $\code$ be a code on $n$ neurons. Let $i,j,k,l,m \in [n]$. Suppose that the canonical form of the neural ideal of $\code$ contains the following pseudo-monomials:
\begin{itemize}
    \item[(i)] $x_ix_k(x_j+1),$ $x_jx_m(x_i+1)$, and $x_k x_m(x_{\ell}+1)$;
    \item[(ii)] $x_{\sigma}$ and $x_{\tau}$, for some $\sigma \subseteq \{i,j,k,\ell\}$ and $\tau \subseteq \{i,j,k,m\}$; and
    \item[(iii)] 
    $x_k (x_j+1) (x_{\ell}+1)$ and 
    $ x_j(x_i+1)(x_k+1)$.
 \end{itemize}
Then $\code$ is not closed-convex.
\end{corollary}

\begin{proof}
Recall that the canonical form is the set of minimal pseudo-monomials in the neural ideal.  Thus, the presence of the pseudo-monomials in $(ii)$ and $(iii)$ imply the properties in rows~4--7 in Table~\ref{tab:criterion}.  

Next, we claim that $x_ix_k(x_j+1) \in {\rm CF}(J_{\code})$ implies the properties in the first row in Table~\ref{tab:criterion}, specifically, $x_i x_j x_k  \notin J_{\code}$.  To see this, note that otherwise the sum $x_ix_k(x_j+1) + x_i x_j x_k = x_i x_k$ would be in $J_{\code}$, which contradicts the fact that $x_ix_k(x_j+1)$ is minimal.  So, by symmetry, the three pseudo-monomials in $(i)$ imply the properties in rows~1--3 in Table~\ref{tab:criterion}.

The corollary now follows from Theorem~\ref{thm:criterion-non-closed-convex-via-RF}.
\end{proof}

\begin{example} \label{ex:3-codes-via-RF-criterion-again}
We return to the codes $C10$ and $C15$. 
Recall that their canonical forms are shown in Example~\ref{ex:canonical-form}.  
Corollary~\ref{cor:criterion-CF} applies to show that these codes are non-closed-convex, where the indices $(i,j,k,\ell, m)$ are as in Table~\ref{tab:ijklm}, and 
$\sigma=\{i,j,\ell\}$ 
and $\tau=\{i,k,m\}$.
\end{example}


\newpage

\section{closed-convexity of sunflower codes} \label{sec:sunflower}
In this section, we focus on a family of codes, called sunflower codes (Definition~\ref{def:sunflower-code}).  These codes were introduced by Jeffs, 
who showed that they have no local obstructions and yet are non-open-convex\footnote{In fact, these codes are minimally non-open-convex with respect to the partial order alluded to earlier in Remark~\ref{rem:partial-order}.}~\cite{jeffs2019sunflowers}.  Nevertheless, we show that these codes are closed-convex and moreover can be realized in $\R^3$ (Theorem~\ref{thm:sunflower-code}).

\begin{definition} \label{def:sunflower-code}
Let $n\geq 2$. The {\em sunflower code}, denoted by $S_n$, 
is the neural code on $2n+2$ neurons that consists
of the following codewords:
\begin{enumerate}
\item the empty codeword $\emptyset$,
\item the ``circle-edge'' codeword $\sigma \cup \{ n + 1\}$, for all nonempty proper subsets $\sigma$ of $[n]$,
\item the ``petal'' codewords $\{n+2\},~\{n+3\},~\dots,~\{2n+2\}$, 
\item the ``petal-end'' codeword $ \{1,2,\dots,i-1\} \cup \{i+1,i+2,\dots, n+1\} \cup \{ n+1+i \}$, for all $1 \leq i \leq n$,
\item the ``polygon'' codeword $\{1,2,\dots,n+1\} \cup \{2n+2\}$, and
\item the ``center'' codeword $\{n+2,n+3,\dots,2n+2\}$. 
\end{enumerate}
\end{definition}
``Petals'' and ``center'' are terms introduced by Jeffs~\cite{jeffs2019sunflowers} (the ``petal'' and ``center'' codewords generate a ``sunflower''). 
The meaning behind these and the other names of the codewords will be shown in the proof of Theorem~\ref{thm:sunflower-code}.

\begin{example}  \label{ex:sunflower-codes}
The first two sunflower codes, with codewords listed in the order given in Definition~\ref{def:sunflower-code}, are as follows:
\begin{align*}
S_2=\{&\emptyset,~   13,~ 23,~4,~ 5,~ 6,~ \mathbf{234},~ \mathbf{135},~ \mathbf{1236},~ \mathbf{456} \}~,\\
S_3=\{&\emptyset,~ 14,~ 24,~  34,~ 124,~134,~ 234,~ 5,~ 6,~ 7,~ 8,~  \mathbf{2345},~  \mathbf{1247},~  \mathbf{1346},~  \mathbf{12348},~  \mathbf{5678} \}~.
\end{align*}
A closed-convex realization of $S_2$ is shown in Figure~\ref{fig:Sunflower-2} (another such a realization is in~\cite[Figure 2]{non-monotonicity}).
\end{example}

\begin{figure}[htb]
\begin{minipage}{0.5\textwidth}
\centering
\begin{tikzpicture}
    \tikzstyle{vertex}=[circle,fill=red, inner sep = 3pt]
	\coordinate (A) at (-2,0);
	\coordinate (B) at (2,0);
	\coordinate (C) at (0,-2);
	\coordinate (D) at (0,0);
	\draw[line width=1mm] (D)--(C);
	\draw[green, line width=1mm] (-2,.225)--(2,.225);
	\draw[red,line width=1mm] (-2,.1)--(0.05,.1);
	\draw[blue,line width=1mm] (2,.35)--(-0.05,.35);
	\draw[line width=1mm] (C)--(A);
	\draw[line width=1mm] (B)--(C);
	\node[label=below:$~$] at (C) {};
	\node[label=left:$U_4$] at (-1,-1) {};
	\node[label=right:$U_5$] at (1,-1) {};
	\node[label=left:$U_6$] at (0.1,-1) {};
	\node[label=$\color{blue}{U_{1}}$] at (1,.325) {};
	\node[label=below:$\color{red}{U_{2}}$] at (-1,.1) {};
	\node[label=above:${\color{green}U_{3}}$] at (-1,.225) {};
\end{tikzpicture}

\end{minipage}%
\begin{minipage}{0.5\textwidth}
\centering
\begin{tikzpicture}
    \tikzstyle{vertex}=[circle, fill=black, inner sep = 3pt]
	\coordinate (A) at (-2,0);
	\coordinate (B) at (2,0);
	\coordinate (C) at (0,-2);
	\coordinate (D) at (0,0);
	\draw[line width=1mm] (D)--(C);
	\draw[line width=1mm] (A)--(B)--(C)--(A);
	\vertex[label=${1236}$] at (D) {};
	\vertex[label=${234}$] at (A) {};
	\vertex[label=${135}$] at (B) {};
	\vertex[label=below:${456}$] at (C) {};
	\node[label=left:$4$] at (-1,-1) {};
	\node[label=right:$5$] at (1,-1) {};
	\node[label=left:$6$] at (0.2,-1) {};
	\node[label=${13}$] at (1,0) {};
	\node[label=${23}$] at (-1,0) {};
\end{tikzpicture}
\end{minipage}
\vspace{-0.2cm}
\caption{A closed-convex realization $\calU=\{U_1,U_2,\dots, U_6\}$ of the sunflower code $S_2$;
depicted on the right are atoms labeled by codewords.}
\label{fig:Sunflower-2}
\vspace{-0.2cm}
\end{figure}
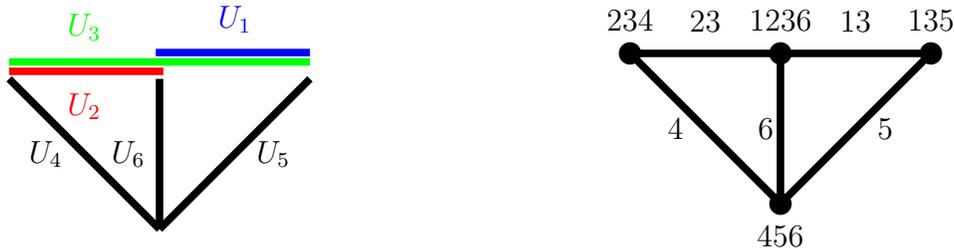

%
%
%
%
%

\begin{theorem}[Closed-convexity of sunflowers] \label{thm:sunflower-code}
The sunflower code $S_2$ is closed-convex in~$\R^2$.
The sunflower code $S_n$, for $n \geq 3$, is closed-convex in $\R^3$.
\end{theorem}
\begin{proof} A closed-convex realization for $S_2$ is shown in Figure~\ref{fig:Sunflower-2}.
For  $n\geq 3$, we construct a closed-convex realization in three steps, as follows.

\noindent
{\bf Step One.}
The restriction of the code $S_n$ to the neurons $\{1,2,\dots, n+1\}$ is 
the code obtained by taking all subsets of $\{1,2,\dots,n\}$ and then adding the neuron $n+1$ to all nonempty subsets.  This restricted code has a unique maximal codeword (namely, $\{1,2, \dots, n+1 \}$, which is the restriction of the ``polygon'' codeword).  Thus, the restricted code can be realized by closed-convex sets in $\mathbb{R}^2$ as a polygon $P$ (with $(2^n-2)$ sides) inscribed in a circle so that each ``circle-edge'' atom lies between the circle and an edge of the polygon~\cite{what-makes}.  This realization is shown for $S_3$ in  Figure~\ref{fig:Sunflower-3}.  

Specifically, $U_{n+1}$ is the filled-in circle, and each of $U_1,U_2,\dots, U_n$ is obtained from this circle by ``slicing off'' some ``circle-edge'' atoms.  Thus, the codewords we have obtained so far are $\{1,2\dots, n+1 \}$ (the subset of the ``polygon'' codeword) and all ``circle-edge'' codewords.

\begin{figure}[ht]
\centering
\begin{tikzpicture}[scale=2.8]
\draw[red, opacity=0.3] (0,1)--({sin(\y)},{cos(\y)})--({sin(2*\y)},{cos(2*\y)})--({sin(3*\y)},{cos(3*\y)})--({sin(4*\y)},{cos(4*\y)})--({sin(5*\y)},{cos(5*\y)})--({sin(6*\y)},{cos(6*\y)});
\draw[red, opacity=0.3] circle(1);
\node[label=$234$] at ({sin(\y)/2+sin(6*\y)/2},{cos(\y)/2+cos(6*\y)/2}){};
\node[label=$34$] at ({sin(2*\y)/2+sin(\y)/2},{cos(\y)/2+cos(2*\y)/2}){};
\node[label=below:$134$] at ({sin(2*\y)/2+sin(3*\y)/2},{cos(2*\y)/2+cos(3*\y)/2}){};
\node[label=below:$14$] at ({sin(3*\y)/2+sin(4*\y)/2},{cos(3*\y)/2+cos(4*\y)/2}){};
\node[label=$124$] at ({sin(4*\y)/2+sin(5*\y)/2},{cos(4*\y)/2+cos(5*\y)/2}){};
\node[label=$24$] at ({sin(5*\y)/2+sin(6*\y)/2},{cos(5*\y)/2+cos(6*\y)/2}){};
\node[label=$1234$] at (0,0){};
\end{tikzpicture}
\caption{Result of Step One for the sunflower code $S_3$.} \label{fig:Sunflower-3}
\end{figure}
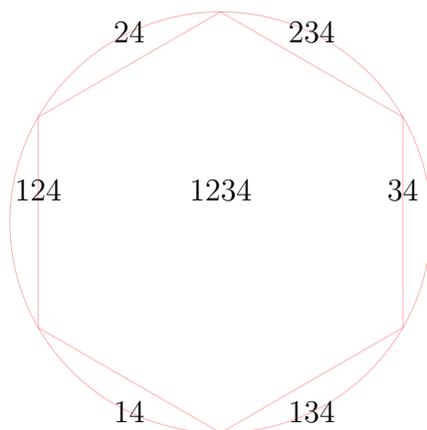

\noindent
{\bf Step Two.}
We pick a point $q $
in a plane parallel to the one in which the polygon $P$ lies, and define $U_{2n+2}$ to be the pyramid with base equal to $P$ and top point $q$.
This is shown in Figure~\ref{fig:Sunflower-3-step-2} for $S_3$.
In this way, the codeword $\{1,2\dots, n+1 \}$ from Step One becomes 
the full ``polygon'' codeword $\{1,2\dots, n+1 \} \cup \{2n+2\}$, and all other codewords from Step One are unaffected.  We also obtain the ``petal'' codeword $\{2n+2\}$.

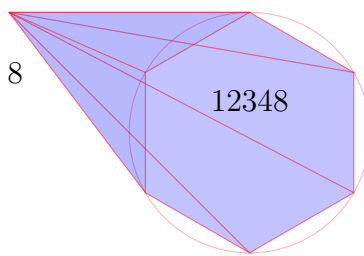
\begin{figure}[ht]
\centering
\begin{tikzpicture}[scale=1.6]
\draw[red, opacity=0.3, fill=blue, fill opacity=0.05] (0,1)--({sin(\y)},{cos(\y)})--({sin(2*\y)},{cos(2*\y)})--({sin(3*\y)},{cos(3*\y)})--({sin(4*\y)},{cos(4*\y)})--({sin(5*\y)},{cos(5*\y)})--({sin(6*\y)},{cos(6*\y)});
\draw[red, opacity=0.3] circle(1);
\draw[red, opacity=0.3, fill=blue, fill opacity=0.2] 
	(0,1) -- ({sin(\y)},{cos(\y)}) -- (-2,1) --(0,1) ;
\draw[red, opacity=0.3, fill=blue, fill opacity=0.2] 
	(-2,1) -- ({sin(\y)},{cos(\y)}) -- ({sin(2*\y)},{cos(2*\y)}) -- (-2,1)  ;
\draw[red, opacity=0.3, fill=blue, fill opacity=0.2] 
	(-2,1) -- ({sin(2*\y)},{cos(2*\y)}) -- ({sin(3*\y)},{cos(3*\y)}) -- (-2,1)  ;
\draw[red, opacity=0.3, fill=blue, fill opacity=0.2] 
	(-2,1) -- ({sin(3*\y)},{cos(3*\y)}) -- ({sin(4*\y)},{cos(4*\y)}) -- (-2,1)  ;
\draw[red, opacity=0.3, fill=blue, fill opacity=0.1] 
	(-2,1) -- ({sin(4*\y)},{cos(4*\y)}) -- ({sin(5*\y)},{cos(5*\y)}) -- (-2,1)  ;
\draw[red, opacity=0.3, fill=blue, fill opacity=0.1] 
	(-2,1) -- ({sin(5*\y)},{cos(5*\y)}) -- ({sin(6*\y)},{cos(6*\y)}) -- (-2,1)  ;
\draw[red, opacity=0.3, fill=blue, fill opacity=0.1] 
	(-2,1) -- (0,1) -- ({sin(6*\y)},{cos(6*\y)}) -- (-2,1)  ;
\node[label=$12348$] at (0,0){};
\node[label=left:$8$] at (-1.7, 0.5){};
\end{tikzpicture}
\caption{Result of Step Two for the sunflower code $S_3$.  The codeword $1234$ from Step One is now $12348$ (the corresponding atom is the hexagon).} \label{fig:Sunflower-3-step-2}
\end{figure}

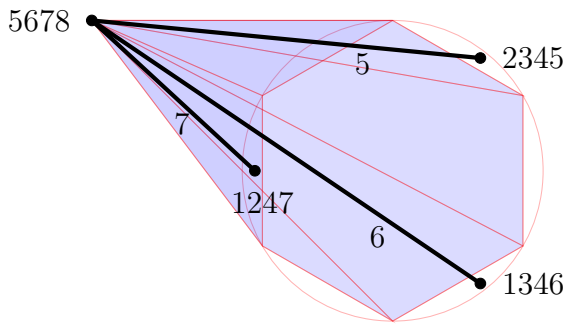
\begin{figure}[ht]
\centering
\begin{tikzpicture}[scale=2]
\draw[red, opacity=0.3, fill=blue, fill opacity=0.05] (0,1)--({sin(\y)},{cos(\y)})--({sin(2*\y)},{cos(2*\y)})--({sin(3*\y)},{cos(3*\y)})--({sin(4*\y)},{cos(4*\y)})--({sin(5*\y)},{cos(5*\y)})--({sin(6*\y)},{cos(6*\y)});
\draw[red, opacity=0.3] circle(1);
\draw[red, opacity=0.3, fill=blue, fill opacity=0.1] 
	(0,1) -- ({sin(\y)},{cos(\y)}) -- (-2,1) --(0,1) ;
\draw[red, opacity=0.3, fill=blue, fill opacity=0.1] 
	(-2,1) -- ({sin(\y)},{cos(\y)}) -- ({sin(2*\y)},{cos(2*\y)}) -- (-2,1)  ;
\draw[red, opacity=0.3, fill=blue, fill opacity=0.1] 
	(-2,1) -- ({sin(2*\y)},{cos(2*\y)}) -- ({sin(3*\y)},{cos(3*\y)}) -- (-2,1)  ;
\draw[red, opacity=0.3, fill=blue, fill opacity=0.1] 
	(-2,1) -- ({sin(3*\y)},{cos(3*\y)}) -- ({sin(4*\y)},{cos(4*\y)}) -- (-2,1)  ;
\draw[red, opacity=0.3, fill=blue, fill opacity=0.1] 
	(-2,1) -- ({sin(4*\y)},{cos(4*\y)}) -- ({sin(5*\y)},{cos(5*\y)}) -- (-2,1)  ;
\draw[red, opacity=0.3, fill=blue, fill opacity=0.1] 
	(-2,1) -- ({sin(5*\y)},{cos(5*\y)}) -- ({sin(6*\y)},{cos(6*\y)}) -- (-2,1)  ;
\draw[red, opacity=0.3, fill=blue, fill opacity=0.1] 
	(-2,1) -- (0,1) -- ({sin(6*\y)},{cos(6*\y)}) -- (-2,1)  ;
\node[label=left:$5678$] at (-2,1){};
\node[label=right:$2345$] at ({sin(\y)/2+sin(6*\y)/2 +0.15 },{cos(\y)/2+cos(6*\y)/2}) {};
\node[label=right:$1346$] at ({sin(2*\y)/2+sin(3*\y)/2 + 0.15 },{cos(2*\y)/2+cos(3*\y)/2}){};
\node[label=below:$1247$] at ({sin(4*\y)/2+sin(5*\y)/2},{cos(4*\y)/2+cos(5*\y)/2}){};
%
\node[label=$5$] at (-0.2, 0.5){};
\node[label=$6$] at (-0.1, -0.65){};
\node[label=$7$] at (-1.4, 0.1 ){};
\draw [ultra thick] (-2,1) -- 	
    ({sin(\y)/2+sin(6*\y)/2 + 0.15}  , {cos(\y)/2+cos(6*\y)/2}) ;
\draw [ultra thick] (-2,1) --  
    ({sin(2*\y)/2+sin(3*\y)/2 + 0.15 },{cos(2*\y)/2+cos(3*\y)/2}) ;
\draw [ultra thick] (-2,1) --  
    ({sin(4*\y)/2+sin(5*\y)/2 -0.05 },{cos(4*\y)/2+cos(5*\y)/2}) ;
\filldraw[black] (-2,1) circle (1pt);
\filldraw[black] ({sin(\y)/2+sin(6*\y)/2 + 0.15} , {cos(\y)/2+cos(6*\y)/2}) circle (1pt);
\filldraw[black] ({sin(2*\y)/2+sin(3*\y)/2 + 0.15 },{cos(2*\y)/2+cos(3*\y)/2}) circle (1pt);
\filldraw[black]     ({sin(4*\y)/2+sin(5*\y)/2 -0.05 },{cos(4*\y)/2+cos(5*\y)/2}) circle (1pt);
\end{tikzpicture}
\caption{Result of Step Three for the sunflower code $S_3$.} \label{fig:Sunflower-3-step-3}
\end{figure}

\noindent
{\bf Step Three.}
To obtain the remaining codewords, we define each of $U_{n+2}, U_{n+3}, \dots, U_{2n+1}$ to be a line segment from $q$ to some point in the corresponding ``circle-edge'' atom.  Thus, these line segments (what Jeffs calls petals~\cite{jeffs2019sunflowers}) all come out of the circle, meeting at a common point (the center~\cite{jeffs2019sunflowers}).
See Figure~\ref{fig:Sunflower-3-step-3} for $S_3$.  
This procedure generates all remaining codewords (without destroying any old ones): the shared endpoint $q$ is the atom of the ``center'' codeword, the other endpoints of the line segments are the atoms of all the ``petal-end'' codewords, and the relative interiors of the line segments are the atoms of the ``petal'' codewords (except $\{2n+2\}$ which was already obtained in Step Two). 

We conclude that the resulting code is 
${\rm code}(\mathcal{U}, \mathbb{R}^3) 
=S_n$.
\end{proof}

\begin{remark} \label{rem:dim-3-vs-2} By 
Theorem~\ref{thm:sunflower-code}, the sunflower codes satisfy
$    \cdim(S_n) \leq 3$, for $n \geq 3$.  However, we do not know whether 
$    \cdim(S_n)$ equals $2$ or $3$.
\end{remark}

\section{Discussion} \label{sec:discussion}
Our work helps clarify how open-convex and closed-convex codes are related.  
Some of our results help unify the theories of open-convexity and closed-convexity: these concepts are the same for nondegenerate codes, although their open and closed embedding dimensions may differ. Thus, for nondegenerate codes, results from open-convexity also hold for and thus can be transferred to closed-convexity, and vice-versa.  

We also introduced a code's nondegenerate embedding dimension, and in the future we would like more results and bounds on this dimension.  Indeed, echoing Cruz {\em et al.}~\cite{cruz2019open}, nondegenerate realizations may be well-suited for applications, as their boundaries do not matter (Theorem~\ref{thm:nondeg-summary}).


Another question for future work, posed in Remark~\ref{rem:partial-order}, concerns the infinite family of codes $\ncfamn$ we constructed: 
Are these codes minimally non-closed-convex with respect to the partial order defined by Jeffs~\cite{jeffs2020morphisms}?  
In fact, no codes have been confirmed (or constructed) to be minimally non-closed-convex but also open-convex.  
In contrast, there are codes that are minimally non-open-convex but also closed-convex: the sunflower codes (Definition~\ref{def:sunflower-code}) form an infinite family of such codes~\cite{jeffs2019sunflowers}. 
Therefore, it would be interesting to find a family of codes that are minimally non-closed-convex but open-convex, or even just one such example. 

The main contribution of our work concerns codes that are not closed-convex (but possibly open-convex).  Notably, we gave the first general criteria for precluding closed-convexity.  
One criterion is built on a novel concept, rigid structures, and we expect this idea to be fruitful in the future (as it already is in~\cite{open-closed-nondeg}).  
To this end, we will need more ways to check whether a code has a rigid structure (Proposition~\ref{prop:path-rigid} is a significant first step).  

In contrast, our other criterion is read directly from a code (or its neural ideal) and, moreover, can often be checked from a code's canonical form.  This criterion, which unifies proofs in previous works, also gives information on which non-maximal codewords can be added to a code without making the code become closed-convex (Remark~\ref{rem:add-codewords}). 
This is remarkable, as analogous results have not yet been stated for open-convexity. 

Going forward, we would like to apply and to improve our new criteria (recall, for instance, the code in Example~\ref{ex:non-monotone-code-rigid} which our criteria can not handle),  
with the aim of further classifying closed-convex codes on $6$ or more neurons.  In turn, such a classification will help elucidate when open-convexity and closed-convexity are essentially the same concept and when these concepts fundamentally differ.




\subsection*{Acknowledgements}
PC, KJ, and JL initiated this research in the 2020 REU in the Department of Mathematics at Texas A\&M
University, supported by the NSF (DMS-1757872).  
AR and AS were supported by the NSF (DMS-1752672).  AR and AS thank R.\ Amzi Jeffs, Caitlin Lienkaemper, and Nora Youngs for helpful discussions.  The authors thank Ola Sobieska for editorial suggestions that improved the exposition, and acknowledge two diligent reviewers for their many detailed suggestions, which strengthened this work.

\bibliography{Bibliography.bib}{}
\bibliographystyle{plain}

\end{document}